\documentclass[oneside]{amsart}
\usepackage{latexsym, amsmath, amssymb, amsthm, mathrsfs}
\usepackage[dvipsnames]{xcolor}

\usepackage{graphicx}
\usepackage{comment}

\usepackage[colorlinks=true,linkcolor=blue,urlcolor=blue, citecolor=blue]%
{hyperref}

\usepackage{tikz-cd}
\usepackage{tikz}
\usetikzlibrary{calc}

\usetikzlibrary{shapes.geometric} 

\usepackage{caption}
\usepackage{subcaption}
\usepackage{nicematrix}
\NiceMatrixOptions{code-for-first-col = \color{red},code-for-first-row = \color{red}}

\usetikzlibrary{patterns}
\usepackage{booktabs}
\usepackage{multirow}
\usepackage{longtable}
\usepackage{changepage}

\usetikzlibrary{patterns}
\usetikzlibrary{arrows.meta}
\definecolor{cof}{RGB}{219,144,71}
\definecolor{pur}{RGB}{186,146,162}
\definecolor{greeo}{RGB}{91,173,69}
\definecolor{greet}{RGB}{52,111,72}

\newtheorem{theorem}{Theorem}[section]

\newtheorem{conjecture}[theorem]{Conjecture}
\theoremstyle{definition}
\newtheorem{definition}[theorem]{Definition}
\newtheorem{example}[theorem]{Example}

\theoremstyle{remark}
\newtheorem{remark}[theorem]{Remark}
\newtheorem{question}[theorem]{Question}

\numberwithin{equation}{theorem}

\title{Integer decomposition property of polytopes}

\author{Sharon Robins}
\address{Department of Mathematics, Simon Fraser University,
	8888 University Drive, Burnaby BC V5A1S6, Canada}
\email{srobins@sfu.ca}
\begin{document}
	\begin{abstract}
		We study the integer decomposition property of lattice polytopes associated to $n$-dimensional smooth complete fans with at most $n+3$ rays. Using the classification of smooth complete fans by Kleinschmidt and Batyrev and a reduction to lower dimensional polytopes, we prove the integer decomposition property for lattice polytopes in this setting.  
	\end{abstract}

	\maketitle

\section{Introduction}

A lattice polytope in $\mathbb{R}^n$ is the convex hull of a finite number of points in the lattice $\mathbb{Z}^n$. A lattice polytope $P$ is said to satisfy the \textbf{integer decomposition property} (or the \textbf{IDP} for short) if for all integers $k\geq1$ and every $x\in kP \cap \mathbb{Z}^n$ there exist $x_1,\ldots,x_k \in P\cap\mathbb{Z}^n$ such that $x=x_1+\cdots+x_k$. Following \cite{MR3679726}, we say a pair of lattice polytopes $(P,Q)$ has the \textbf{IDP} if
\[(P\cap \mathbb{Z}^n)+(Q\cap \mathbb{Z}^n)=(P+Q)\cap \mathbb{Z}^n,\]
where the addition is the Minkowski sum. We can see that $P$ has the IDP if and only if every pair $(P,kP)$ has the IDP for all integers $k\geq1$.

Since Oda \cite{oda2008problems} introduced the question of identifying the polytopes with the IDP, these polytopes have captured the attention of researchers in integer programming, graph theory, commutative algebra, and toric geometry. The IDP does not hold for any general pair of lattice polytopes $P$ and $Q$, even in the special case when $P=Q$ (see \cite[Example 2.2.11]{MR2810322}). However, we know that the IDP holds in the following situations: 
\begin{itemize}
	\item Let $R$ be a $n$-dimensional lattice polytope. Let $P=kdR,Q=ldR$ for $k,l\in \mathbb{N}$ and $d\geq n-1$  \cite{EW91} (see also \cite{LTZ93}, \cite[Theorem 1.3.3]{BGT97}).
	\item Let $R$ be a unimodular simplex or  parallelepiped or zonotope or a centrally symmetric 3-dimensional smooth polytope. Let $P=kR,Q=lR$ for $k,l\in \mathbb{N}$ \cite[Proposition 1.2, Theorem 1.4]{MR3962856}.
	\item Let $R$ be a lattice polytope of dimension $n$ and suppose every edge has lattice length $2n(n+1)$. Let $P=kR,Q=lR$ for $k,l\in \mathbb{N}$ \cite[Corollary 6]{MR3679726}. We should point out that this bound $2n(n+1)$ is an improvement over the prior one $4n(n+1)$ in \cite[Theorem 1.3]{Gubeladze_2012}. 
\end{itemize}

The above results suggest that if $P,Q$ are a big enough dilation of a lattice polytope $R$, then the pair $(P,Q)$ has the IDP. Apart from the above cases, the following conjecture has received special interest.

\begin{conjecture}[Oda’s Conjecture]{\cite[Problem 3]{oda2008problems}}\label{Oda'sconjecture}
	Let $P,Q\subseteq \mathbb{R}^n$ be two lattice polytopes. Suppose that $P$ is smooth and its normal fan refines the one of $Q$. Then the pair $(P,Q)$ has the IDP.
\end{conjecture}

Fakhruddin \cite{2002} and Ogata \cite{MR2267321} give two independent proofs of IDP, for the condition when $P$ is a smooth lattice polygon and the normal fan of $P$ is a refinement of the normal fan of $Q$. Even if we drop the assumption that $P$ is smooth, IDP holds \cite{MR2398828}. But if we drop the assumption on the normal fan, the IDP does not hold in general (see Figure \ref{fig:cone}). In Figure \ref{fig:fan}, we draw the normal fan of $P+Q$ and a minimal smooth refinement. We can see that it has 8 rays. This motivates our research question to study the lowest possible cases that fail to satisfy the IDP. Notice that in 2 dimensions one has a unique minimal smooth refinement, however in higher dimensions this is no longer true.

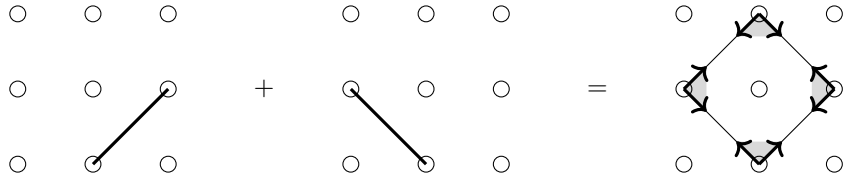
\begin{figure}[ht]
	\begin{subfigure}{.3\textwidth}
		\centering    
		\begin{tikzpicture}
			
			\foreach \x in {0,...,2}{       \foreach \y in {0,...,2}{  
					\draw (\x,\y) circle(3pt);
				}
			}
			\draw[very thick,-] (1,0)--(2,1); 
		\end{tikzpicture}
		
	\end{subfigure} \hfill
	+
	\begin{subfigure}{.3\textwidth}
		\centering    
		\begin{tikzpicture}
			
			\foreach \x in {0,...,2}{       \foreach \y in {0,...,2}{  
					\draw (\x,\y) circle(3pt);
				}
			}
			\draw[very thick,-] (1,0)--(0,1); 
		\end{tikzpicture}
	\end{subfigure} \hfill
	=
	\begin{subfigure}{.3\textwidth}
		\centering    
		\begin{tikzpicture}

			\fill[gray!30] (1,0)--(0.7,0.3)--(1.3,0.3);
			\draw[very thick,->] (1,0)--(0.7,0.3); 
			\draw[very thick,->] (1,0)--(1.3,0.3);

			\fill[gray!30] (2,1)--(1.7,0.7)--(1.7,1.3);
			\draw[very thick,->] (2,1)--(1.7,0.7); 
			\draw[very thick,->] (2,1)--(1.7,1.3); 
			
			\fill[gray!30] (1,2)--(1.3,1.7)--(0.7,1.7);
			\draw[very thick,->] (1,2)--(1.3,1.7); 
			\draw[very thick,->] (1,2)--(0.7,1.7);
			
			\fill[gray!30] (0,1)--(0.3,0.7)--(0.3,1.3);
			\draw[very thick,->] (0,1)--(0.3,0.7); 
			\draw[very thick,->] (0,1)--(0.3,1.3);
			
			\draw[-] (1,0)--(0,1); 
			\draw[-] (1,0)--(2,1); 
			\draw[-] (2,1)--(1,2); 
			\draw[-] (0,1)--(1,2); 
			
			\foreach \x in {0,...,2}{       \foreach \y in {0,...,2}{  
					\draw (\x,\y) circle(3pt);
				}
			}
		\end{tikzpicture}

	\end{subfigure}    
	\caption{ An example of polytopes that do not have IDP}
	\label{fig:cone}
\end{figure}


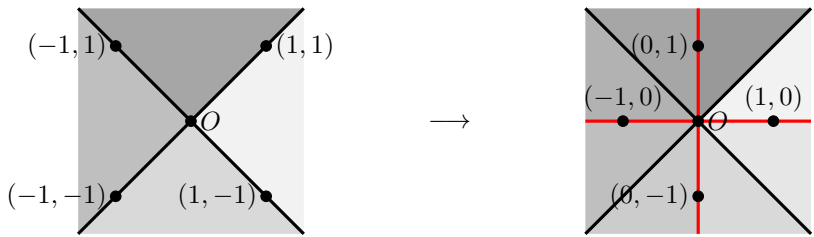
\begin{figure}[ht]
	\begin{subfigure}{.45\textwidth}
		\centering  
		\begin{tikzpicture}[scale=0.5]
			\coordinate (O) at (0,0);
			\coordinate (r1) at (2,2);
			\coordinate (R1) at (3,3);
			\coordinate (r2) at (2,-2);
			\coordinate (R2) at (3,-3);
			\coordinate (r3) at (-2,-2);
			\coordinate (R3) at (-3,-3);
			\coordinate (r4) at (-2,2);
			\coordinate (R4) at (-3,3);
			\coordinate (R5) at (3,0);
			\coordinate (R6) at (0,-3);
			\coordinate (R7) at (-3,0);
			\coordinate (R8) at (0,3);
			\coordinate (r5) at (2,0);
			\coordinate (r6) at (0,-2);
			\coordinate (r7) at (-2,0);
			\coordinate (r8) at (0,2);

			\fill[gray!10] (O)--(R1)--(R2);
			\fill[gray!30] (O)--(R2)--(R3);
			\fill[gray!50] (O)--(R3)--(R4);
			\fill[gray!70] (O)--(R4)--(R1);
			\draw[very thick,-] (O)--(R1);
			\draw[very thick,-] (O)--(R2);
			\draw[very thick,-] (O)--(R3);
			\draw[very thick,-] (O)--(R4);

			\filldraw[black] (O) circle (4pt) node[anchor=west] {$O$};
			\filldraw[black] (r1) circle (4pt) node[anchor=west] {$(1,1)$};
			\filldraw[black] (r2) circle (4pt) node[anchor=east] {$(1,-1)$};
			\filldraw[black] (r3) circle (4pt) node[anchor=east] {$(-1,-1)$};
			\filldraw[black] (r4) circle (4pt) node[anchor=east] {$(-1,1)$};
		\end{tikzpicture}
	\end{subfigure}
	\hfill
	$\longrightarrow$
	\begin{subfigure}{.45\textwidth}
		\centering  
		\begin{tikzpicture}[scale=0.5]
			
			\fill[gray!10] (O)--(R1)--(R5);
			\fill[gray!20] (O)--(R5)--(R2);
			\fill[gray!30] (O)--(R2)--(R6);
			\fill[gray!40] (O)--(R6)--(R3);
			\fill[gray!50] (O)--(R3)--(R7);
			\fill[gray!60] (O)--(R7)--(R4);
			\fill[gray!70] (O)--(R4)--(R8);
			\fill[gray!80] (O)--(R8)--(R1);
			
			\draw[red,very thick,-] (O)--(R5); 
			\draw[very thick,-]
			(O)--(R1);
			\draw[red,very thick,-] (O)--(R6);
			\draw[very thick,-] (O)--(R2);
			\draw[red,very thick,-] (O)--(R7);
			\draw[very thick,-] (O)--(R3);
			\draw[red,very thick,-] (O)--(R8);
			\draw[very thick,-] (O)--(R4);
			
			\filldraw[black] (r5) circle (4pt) node[anchor=south] {$(1,0)$};
			\filldraw[black] (r6) circle (4pt) node[anchor=north east] {$(0,-1)$};
			\filldraw[black] (r7) circle (4pt) node[anchor=south east] {$(-1,0)$};
			\filldraw[black] (r8) circle (4pt) node[anchor=south east] {$(0,1)$};
			
		\end{tikzpicture}
		
	\end{subfigure}

	\caption{Refining the normal fan into a smooth fan}
	\label{fig:fan}
\end{figure}

In general, from a given complete normal fan in $\mathbb{R}^n$ we can construct a simplicial fan by triangulating the maximal cones. By subdividing those cones further, we can construct a smooth complete fan with $n+k$ rays. We can then ask the following question:

\begin{question}\label{question1.3}
	Let $P,Q$ be two lattice polytopes and $\Sigma$ be the normal fan of $P+Q$. Suppose we refined $\Sigma$ to a minimal smooth complete fan $\Sigma'$. Can we find a bound on $k$ so that if $\Sigma'$ has only $n+k$ rays then the pair $(P,Q)$ always has the IDP?
\end{question}

Alternatively, we can start from a smooth complete fan $\Sigma$ in $\mathbb{R}^n$ with $n+k$ rays and study all lattice polytopes cut out by half spaces from all or some of these rays. All of these lattice polytopes can be studied in terms of convex support functions on $\Sigma$. See Section 2 for details. Hence, we can rephrase Question \ref{question1.3} as follows.

\begin{question}\label{question1.4}
	Let $\Sigma$ be a smooth complete fan in $\mathbb{R}^n$ with $n+k$ rays and let $P,Q$ be two lattice polytopes associated with two convex support functions for $\Sigma$. Does the pair $(P,Q)$ have the IDP?
\end{question}

The answer to Question \ref{question1.4} is known to be positive when $k=1$ or $2$, see Theorem \ref{IDP4}. In this paper, we will prove that the IDP holds for the case $k=3$. 

\begin{theorem}\label{IDPtheorem}
	If $P,Q$ are two lattice polytopes whose coarsest common refinement of normal fans has at most $n+3$ rays in $\mathbb{R}^n$, then the pair $(P,Q)$ has the IDP.
\end{theorem}

It is worth noting that normal fan of a smooth lattice polytope in $\mathbb{R}^n$ with $n+k$ facets is a smooth complete fan in $\mathbb{R}^n$ with $n+k$ rays. Hence, as a  corollary to Theorem \ref{IDPtheorem}, we observe that Conjecture \ref{Oda'sconjecture} is true if $P$ is a smooth lattice polytope in $\mathbb{R}^n$ with at most $n+3$ facets. In particular, this implies that every smooth lattice polytope with at most $n+3$ facets has the IDP. An application of the result can be seen in \cite[Theorem 2.1]{Rob22}.

On the other hand, one can ask what the answer for Question 1.3 is in the particular situation where $n=2$. Our result shows that as long as $\Sigma$ has 5 or fewer rays IDP holds. However, the example in Figure \ref{fig:cone} shows that there exists a smooth complete fan in $\mathbb{R}^2$ with 8 rays which has a negative answer to Question \ref{question1.4}. We ask the following question:

\begin{question}
	Let $P, Q$ be two lattice polygons and suppose that the normal fan of $P+Q$ can be refined to a smooth fan in $\mathbb{R}^2$ of 6 or 7 rays. Does the pair $(P, Q)$ have the IDP? 
\end{question}

We believe the answer is yes, but do not have a proof. We intend to address this question in future research.

This paper is organized as follows. In Section 2, we give the preliminary details about fans and lattice polytopes. We also recall the classification of smooth complete fans in $\mathbb{R}^n$ with $n+3$ rays. We prove the Theorem \ref{IDPtheorem} in Section 3. 

\subsection*{Acknowledgements}
I am thankful to Nathan Ilten for a number of helpful conversations, as well as comments on preliminary versions of this paper. I would also like to express my gratitude to the anonymous referee for their insightful comments.

\section{Preliminaries}
In this section, we fix notation and recall some basic facts about fans. For details see {\cite[Chapter 3]{MR2810322}}. We also recall the classification of smooth complete fans in $\mathbb{R}^n$ with $n+3$ rays by Batyrev \cite{batyrev1991}.

All cones in this work are assumed to be rational polyhedral, i.e., subsets of $\mathbb{R}^n$ of the form $\sigma=\{\sum_{i=1}^{m}\lambda_iu_i: \lambda_i\geq0\}$ for some $u_1,\ldots,u_m\in \mathbb{Z}^n$. A fan $\Sigma$ in $\mathbb{R}^n$ is said to be \textbf{complete} if the union of all cones in $\Sigma$ is $\mathbb{R}^n$. Moreover, $\Sigma$ is said to be \textbf{smooth} if every cone $\sigma \in \Sigma$ is generated by part of a basis of $\mathbb{Z}^n$. Let $\Sigma$ be a smooth complete fan in $\mathbb{R}^n$ with $n+m$ rays. We denote the rays of $\Sigma$ by $\Sigma(1)$. Given a ray $\rho \in \Sigma(1)$, we denote by $u_{\rho}$ the \textbf{primitive generator} in $\mathbb{Z}^n$ of $\rho$. 

Two fans $\Sigma_1$ and $\Sigma_2$ are said to be \textbf{unimodular equivalent} if there exists a unimodular transformation $L:\mathbb{R}^n \rightarrow \mathbb{R}^n$ that preserves $\mathbb{Z}^n$ and maps the cones of $\Sigma_1$ bijectively to the cones of $\Sigma_2$. The unimodularity condition is motivated by algebraic geometry: the toric varieties associated with the unimodularly equivalent fans are isomorphic. We will consider fans up to unimodular equivalence. For every smooth complete fan, there is a unimodularly equivalent fan with one of the maximal cones generated by the standard basis $\{e_1,\ldots,e_n\}$ of $\mathbb{R}^n$. 

Let $\Sigma$ be a smooth complete fan in $\mathbb{R}^n$ with $n+m$ rays. Let $A$ be the $(n+m)\times n$ matrix whose rows represent the primitive ray generators of $\Sigma$ in terms of the standard basis $\{e_1,\ldots,e_n\}$ of $\mathbb{R}^n$. If necessary by rearranging the ordering of rays, $A$ has the form $\left[   
\begin{array}{c}
	I      \\
	\hline
	B
\end{array}
\right]$
where $I$ is the identity matrix and $B$ is an integer matrix. We denote by $\{f_1,\ldots,f_{n+m}\}$ the standard basis of $\mathbb{R}^{n+m}$. If we are not specifying the ordering of the rays, then we alternatively denote the standard basis of $\mathbb{R}^{n+m}$ by $\{f_{\rho}:\rho\in \Sigma(1)\}$. Then it is easy to see that $\{Ae_1,\ldots,Ae_n\}$ can be extended to a $\mathbb{Z}$-basis of $\mathbb{Z}^{n+m}$ by adding the vectors $f_{n+1},\ldots,f_{n+m}$. Therefore, we can identify $\mathbb{Z}^{m+n}/A\mathbb{Z}^n$ with $\mathbb{Z}^m$, and we get the following short exact sequence:

\begin{equation}\label{exactseq}
	0{\longrightarrow} \mathbb{Z}^n \overset{A}\longrightarrow \mathbb{Z}^{n+m} \longrightarrow \mathbb{Z}^{m} \rightarrow 0.
\end{equation}

To any $h\in \mathbb{Z}^{n+m}$, we can associate a (possibly empty) polytope $P(A,h) \subseteq \mathbb{R}^n$ defined by 
\[P(A,h)=\{ x \in \mathbb{R}^n:Ax\geq -h\}.\]
If $h-h'=Ax$ for some $x\in \mathbb{Z}^n$, then $P(A,h)$ is a translation of $P(A,h')$ by an integral vector. Since translations by integral vectors preserve the IDP, it is enough to study the polytopes up to the image of $A\mathbb{Z}^n$.

A \textbf{support function} on $\Sigma$ is a function $\varphi:\mathbb{R}^n \rightarrow \mathbb{R} $ that is linear on each cone of $\Sigma$ and is $\mathbb{Z}$-valued on $\mathbb{Z}^n$. To any $h\in \mathbb{Z}^{n+m}$, we define a function $\varphi_{h}$ such that
\begin{equation}\label{support}
	\varphi_{h}(u_{\rho})=-h_{\rho} \text{ for all } \rho\in \Sigma(1),
\end{equation}
where $h_{\rho}$ is the $\rho^{\text{th}}$
coordinate of $h$. Since $\Sigma$ is smooth, $\varphi_h$ can be extended uniquely to a support function of $\Sigma$. We say $\varphi_h$ is convex if 
\[\varphi(x+y)\geq \varphi(x)+\varphi(y)\]
for all $x,y\in \mathbf{R}^n$. Moreover, we say $\varphi_h$ is strictly convex if 
\[\varphi(x+y)> \varphi(x)+\varphi(y)\] 
for all $x,y\in \mathbf{R}^n$ not lying on the same maximal cone of $\Sigma$. If $h\in \mathbb{Z}^{n+m}$ defines a convex support function $\varphi_h$, then $P(A,h)$ is a lattice polytope. Conversely, if $Q$ is a lattice polytope whose normal fan can be refined into a smooth complete fan $\Sigma$, then there exists a convex support function $\varphi_h$ on $\Sigma$ such that $P(A,h)$ is equal to $Q$. If $\varphi_h,\varphi_{h'}$ are two convex support functions on $\Sigma$, then
\begin{equation}\label{sumpolytopes}
	P(A,h+h')=P(A,h)+P(A,h'),
\end{equation}    
where the addition of polytopes is the Minkowski sum.

In terms of support functions, we can state  Conjecture \ref{Oda'sconjecture} as follows.

\begin{conjecture}
	Let $\Sigma$ be a smooth projective fan, and $A$ be the matrix representation of $\Sigma$. Suppose that $\varphi_h$ is strictly convex, and $\varphi_{h'}$ is convex on $\Sigma$. Then the pair $(P(A,h),P(A,h'))$ has the IDP.
\end{conjecture}

Now we will discuss the definition of primitive collections introduced by Batyrev, which makes the classifications and computations easier.

\begin{definition}\label{primitive}
	Let $\Sigma$ be a fan.
	A subset $\mathscr{P}=\{\rho_1,\rho_2,\ldots,\rho_k\} \subset \Sigma(1)$ is called a primitive collection if $\mathscr{P}$ is not contained in a single cone of $\Sigma$, but every proper subset is. Let $\mathscr{P}$ be a primitive collection and $\sigma\in\Sigma$ be the cone of the smallest dimension containing $u_{\rho_1}+\cdots+u_{\rho_k}$. Then there exists a unique expression
	\begin{equation}\label{primitiverel}
		u_{\rho_1}+\cdots+u_{\rho_k}=\sum_{\rho\in \sigma(1)}c_{\rho}u_{\rho}, \text{~~~}c_{\rho}\in\mathbb{Z}_{>0}. 
	\end{equation}
	The equation (\ref{primitiverel}) is called the primitive relation of $\mathscr{P}$. 
\end{definition}

\begin{definition}
	A fan is called a splitting fan if there is no intersection between any two primitive collections.
\end{definition}

It is easy to check whether $\varphi_h$ is convex or not using primitive collections.

\begin{theorem}[{\cite[Theorem 2.15]{batyrev1991}}]\label{theoremconevex}
	Let $\Sigma$ be a smooth projective fan in $\mathbb{R}^n$ and $h\in \mathbb{Z}^{(n+m)}$. Then
	$\varphi_h$ is convex if and only if it satisfies
	\[\varphi_{h}(u_{\rho_{1}}+\cdots+u_{\rho_{k}})\geq \varphi_{h}(u_{\rho_{1}})+\cdots+\varphi_{h}(u_{\rho_{k}})\]
	for all primitive collections $\mathscr{P}=\{\rho_1,\ldots,\rho_k\}$ of $\Sigma$. Similarly, $\varphi_h$ is strictly convex 
	if and only if it satisfies
	\[\varphi_{h}(u_{\rho_{1}}+\cdots+u_{\rho_{k}})> \varphi_{h}(u_{\rho_{1}})+\cdots+\varphi_{h}(u_{\rho_{k}})\]
	for all primitive collections $\mathscr{P}=\{\rho_1,\ldots,\rho_k\}$ of $\Sigma$.
\end{theorem}

We also point out that these results hold in a broader context (see {\cite[Theorem 1.4]{CvR09}} and  {\cite[Theorem 6.4.9]{MR2810322}}). For the rest of this section we discuss the classification of smooth complete fans in $\mathbb{R}^n$ with at most $n+3$ rays.

\begin{example}\label{rank1}
	Let $\Sigma$ be a smooth complete fan in $\mathbb{R}^n$ with $n+1$ rays. It is easy to see that there is only one unimodular equivalent class. The ray generators of $\Sigma$ can be written in the matrix $A$ as follows:
	\NiceMatrixOptions{code-for-first-col = \color{red},code-for-first-row = \color{red}}
	\newcommand{\blue}{\color{blue}}
	\newcommand{\green}{\color{teal}}
	\newcommand{\magenta}{\color{magenta}}
	\setcounter{MaxMatrixCols}{20}
	
	\begin{equation}
		A=
		\begin{pNiceMatrix}[small,first-row,first-col,margin,hvlines]
			& & \Ldots[line-style={solid,<->},shorten=0pt]^{n \text{ columns}} \\
			\rho_1& \Block{3-3}<\LARGE>{\blue I} \\
			\vdots&&&\\ 
			\rho_n&&\\
			\rho_{n+1}&\blue -1 & \blue \Cdots&\blue -1\\
		\end{pNiceMatrix}
		\label{matrix}.
	\end{equation}
	\\
	It has only one primitive collection that consists of all rays and the primitive relation is given by
	\[u_{\rho_1}+\cdots+u_{\rho_{n+1}}=0.\]
	The support function $\varphi_h$ is convex if and only if 
	\[\varphi_h(u_{\rho_1}+\cdots+u_{\rho_{n+1}})\geq \varphi_h(u_{\rho_1})+\cdots+\varphi_h(u_{\rho_{n+1}}).\]
	Simplifying the inequality, we get
	\[\varphi_h(0)\geq -\sum_{\rho\in \Sigma(1)} h_{\rho},\]
	i.e., $h$ is convex support function if and only if $\sum_{\rho \in \Sigma(1)} h_{\rho} \geq 0$.
\end{example}

Kleinschmidt classified smooth complete $n$-dimensional fans with $n+2$ rays. It is known that a smooth complete fan in $\mathbb{R}^n$ of at most $n+2$ rays is a splitting fan.

\begin{theorem}[{\cite[Theorem 1]{Kleinschmidt1988}}]\label{rank2}
	Let $\Sigma$ be a smooth complete fan $\mathbb{R}^n$ with $n+2$ rays. In our setting, the generators of the smooth complete fan in $\mathbb{R}^n$ with $n+2$ rays can be written in the matrix $A$ as follows:

	\NiceMatrixOptions{code-for-first-col = \color{red},code-for-first-row = \color{red}}
	\newcommand{\blue}{\color{blue}}
	\newcommand{\green}{\color{green}}
	\newcommand{\magenta}{\color{magenta}}
	\setcounter{MaxMatrixCols}{20}
	
	\begin{equation}
		A=
		\begin{pNiceMatrix}[small,first-row,first-col,margin,hvlines]
			& & \Ldots[line-style={solid,<->},shorten=0pt]^{n \text{ columns}} \\
			\rho_1& \Block{3-3}<\LARGE>{\blue I} & & &\Block{3-3}<\LARGE>{ \blue 0}&&\\
			\vdots&&&&&&\\ 
			\rho_k&&&&&&\\
			\rho_{k+1}& \Block{3-3}<\LARGE>{\blue 0} & & &\Block{3-3}<\LARGE>{ \blue I}&&\\
			\vdots&&&&&&\\ 
			\rho_n&&&&&&\\
			\rho_{n+1}&\blue -1& \blue \Cdots&\blue -1&\blue 0&\blue \Cdots&\blue 0\\
			\rho_{n+2}&\blue a_1& \blue \Cdots&\blue a_k&\blue -1&\blue \Cdots&\blue -1\\
		\end{pNiceMatrix}
		\label{matrix2}
	\end{equation}
	\\
	where $1\leq k\leq n-1$ and $a_k\leq \cdots \leq a_1\leq0$. The primitive collections are given by $\{\rho_1,\ldots,\rho_k,\rho_{n+1}\}$ and $\{\rho_{k+1},\ldots,\rho_n,\rho_{n+2}\}$.
\end{theorem}

\begin{theorem}[{\cite[Corollary 4.2]{MR2551605}}]\label{IDP4}
	If $\Sigma$ is a smooth complete splitting fan and $\varphi_h,\varphi_{h'}$ are convex support functions of $\Sigma$, then the pair $(P(A,h),P(A,h'))$ has the IDP. In particular, if $\Sigma$ is a smooth complete fan in $\mathbb{R}^n$
	with at most $n+2$ rays then the statement holds.
\end{theorem}

Batyrev classified smooth complete $n$-dimensional fans with $n+3$ rays in terms of primitive collections. Batyrev showed that the number of primitive collections of its rays is 3 or 5 \cite[Theorem 5.7]{batyrev1991}. The case of 3 primitive collections is a splitting fan. The case of 5 primitive collections is as in the following theorem.

\begin{figure}
	\centering
	\begin{tikzpicture}[scale=3]
		\node[draw, thick, black, rotate=0, minimum size=3cm, regular polygon, regular polygon sides=5] (pol) at (0,0) {}; 
		
		\filldraw[black] (pol.corner 1) circle (0.5pt) node[anchor=south, outer sep=5pt] {$X_0$};
		\filldraw[black] (pol.corner 2) circle (0.5pt) node[anchor=east, outer sep=5pt] {$X_1$};
		\filldraw[black] (pol.corner 3) circle (0.5pt) node[anchor=east, outer sep=5pt] {$X_2$};
		\filldraw[black] (pol.corner 4) circle (0.5pt) node[anchor=west, outer sep=5pt] {$X_3$};
		\filldraw[black] (pol.corner 5) circle (0.5pt) node[anchor=west, outer sep=5pt] {$X_4$};
	\end{tikzpicture}
	\caption{}
	\label{fig:pent}
\end{figure}

\begin{theorem}[{\cite[Theorem 6.6]{batyrev1991}}]\label{Batyrev}
	Let $\mathscr{X}_{\alpha}=X_{\alpha}\cup X_{\alpha+1}$, where $\alpha \in \mathbb{Z}/{5\mathbb{Z}}$,
	\begin{align*}
		X_0&=\{v_1,\ldots,v_{p_0}\},    &   X_1&=\{y_1,\ldots,y_{p_1}\}, &     X_2=\{z_1,\ldots,z_{p_2}\},\\
		X_3&=\{t_1,\ldots,t_{p_3}\},    & X_4&=\{u_1,\ldots,u_{p_4}\},
	\end{align*}
	and $p_0+p_1+p_2+p_3+p_4=n+3$. It is convenient to use a picture of a pentagon with the vertices $\iota e^{\iota 2\pi  \alpha /5}$ in the complex plane (see Figure \ref{fig:pent}). Then any smooth complete $n$-dimensional fan $\Sigma$ with the set of generators $\bigcup X_{\alpha}$ and five primitive collections $\mathscr{X}_{\alpha}$ can be described up to a symmetry of the pentagon by the following primitive relations with nonnegative integral coefficients $c_2,\ldots,c_{p_2},b_1,\ldots,b_{p_3}$:
	\begin{equation*}
		\begin{split}
			v_1+\cdots +v_{p_0}+y_1+\cdots+y_{p_1}&=c_2z_2+\cdots+c_{p_2}z_{p_2}+(b_1+1)t_1+\cdots+(b_{p_3}+1)t_{p_3},\\
			y_1+\cdots+y_{p_1}+z_1+\cdots+z_{p_2}&=u_1+\cdots+u_{p_4},\\
			z_1+\cdots+z_{p_2}+t_1+\cdots+t_{p_3}&=0,\\
			t_1+\cdots+t_{p_3}+u_1+\cdots+u_{p_4}&=y_1+\cdots+y_{p_1},\\
			u_1+\cdots+u_{p_4}+v_1+\cdots +v_{p_0}&=c_2z_2+\cdots+c_{p_2}z_{p_2}+b_1t_1+\cdots+b_{p_3}t_{p_3}.
		\end{split}
	\end{equation*}
\end{theorem}

In Theorem \ref{Batyrev} we can identify the set 
\[\{v_1,\ldots,v_{p_0},u_2,\ldots,u_{p_4},y_2,\ldots,y_{p_1},t_1,\ldots,t_{p_3},z_2,\ldots,z_{p_2}\}\]
with the standard basis $\{e_1,\ldots,e_n\}$ of $\mathbb{Z}^n$. Thus $z_1,u_1,y_1$ are defined by
\begin{equation*}
	\begin{split}
		z_1&=-z_2-\cdots-z_{p_2}-t_1-\cdots-t_{p_3},\\
		u_1&=-v_1-\cdots-v_{p_0}-u_2-\cdots-u_{p_4}+b_1t_1+\cdots+b_{p_3}t_{p_3}+c_2z_2+\cdots+c_{p_2}z_{p_2},\\
		y_1&=-v_1-\cdots-v_{p_0}-y_2-\cdots-y_{p_1}+(b_1+1)t_1+\cdots+(b_{p_3}+1)t_{p_3}+\\
		& \quad c_2z_2+\cdots+c_{p_2}z_{p_2}.
	\end{split}    
\end{equation*}

Let $\Sigma$ be a fan as in Theorem \ref{Batyrev}. By arranging the primitive ray generators in rows of a matrix we will get a matrix $A$ of the following form

\NiceMatrixOptions{code-for-first-col = \color{red},code-for-first-row = \color{red}}
\newcommand{\blue}{\color{blue}}
\newcommand{\green}{\color{teal}}
\newcommand{\magenta}{\color{magenta}}
\setcounter{MaxMatrixCols}{20}

\begin{equation}
	A=
	\begin{pNiceMatrix}[small,first-row,first-col,margin,hvlines]
		& & \Ldots[line-style={solid,<->},shorten=0pt]^{n \text{ columns}} \\
		v_1& \Block{3-3}<\LARGE>{\blue I} & & &\Block{3-3}<\LARGE>{ \blue 0}&&&\Block{3-3}<\LARGE>{\blue 0}&&&\Block{3-3}<\LARGE>{\magenta 0}&&&\Block{3-3}<\LARGE>{ \magenta 0}&&\\
		\vdots &&  & &&&&&&&&&&  \\
		v_{p_0}&& & &  \\
		u_1& \blue -1& \blue \Cdots&\blue -1&\blue -1&\blue \Cdots&\blue -1&\blue 0&\blue \Cdots&\blue 0&\magenta b_1& \magenta \Cdots&\magenta b_{p_3}& \magenta c_2&\magenta \Cdots&\magenta c_{p_2}\\
		u_2&\Block{3-3}<\LARGE>{\blue 0} & & &\Block{3-3}<\LARGE>{\blue I}&&&\Block{3-3}<\LARGE>{\blue 0}&&&\Block{3-3}<\LARGE>{\magenta 0}&&&\Block{3-3}<\LARGE>{\magenta 0}&&\\
		\vdots &&  & &&&&&&&&&&  \\
		u_{p_4}&& & &  \\
		y_1&\blue -1&\blue \Cdots&\blue -1 &\blue 0&\blue \Cdots& \blue 0&\blue -1&\blue \Cdots&\blue -1&\magenta (b_1+1)&\magenta \Cdots&\magenta (b_{p_3}+1)&\magenta c_2&\magenta \Cdots&\magenta c_{p_2}\\
		y_2&\Block{3-3}<\LARGE>{\blue 0} & & &\Block{3-3}<\LARGE>{\blue 0}&&&\Block{3-3}<\LARGE>{\blue I}&&&\Block{3-3}<\LARGE>{\magenta 0}&&&\Block{3-3}<\LARGE>{\magenta 0}&&\\
		\vdots &&  & &&&&&&&&&&  \\
		y_{p_1}&& & &  \\
		t_1&\Block{3-3}<\LARGE>{0} & & &\Block{3-3}<\LARGE>{0}&&&\Block{3-3}<\LARGE>{0}&&&\Block{3-3}<\LARGE>{\green I}&&&\Block{3-3}<\LARGE>{\green 0}&&\\
		\vdots &&  & &&&&&&&&&&  \\
		t_{p_3}&& & &  \\
		z_2&\Block{3-3}<\LARGE>{0} & & &\Block{3-3}<\LARGE>{0}&&&\Block{3-3}<\LARGE>{0}&&&\Block{3-3}<\LARGE>{\green 0}&&&\Block{3-3}<\LARGE>{\green I}&&\\
		\vdots &&  & &&&&&&&&&&  \\
		z_{p_2}&& & &&&&  \\
		z_1&0 & \Cdots& 0 & 0&\Cdots&0&0&\Cdots&0&\green -1&\green \Cdots&\green -1&\green -1&\green \Cdots&\green -1\\
	\end{pNiceMatrix}
	\label{matrix3}.
\end{equation}

Note that $A$ has the block matrix form $\left[   
\begin{array}{c|c}
	F   & G   \\
	\hline
	0   & H 
\end{array}
\right]$.
Here $H$ corresponds to the matrix  of a smooth complete fan in $\mathbb{R}^{p_2+p_3-1}$ with $p_2+p_3$ rays (see Example \ref{rank1}).

\section{Main result}

The goal of this section is to prove Theorem \ref{IDPtheorem}. Our results enlarge the collection of polytopes satisfying the integer decomposition property. Before proving Theorem \ref{IDPtheorem}, we introduce some notations and outline the general strategy of the proof. 

Let $A$ be a matrix as in (\ref{matrix}). We will study the polytopes of the form $P(A,h)$, where $h\in \mathbb{Z}^{n+3}$ defines a convex support function. We have seen that if $h'=h+Ax$ for some $x\in\mathbb{Z}^n$, then $P(A,h')$ is a translation of $P(A,h)$ by a lattice vector. Hence, it is enough to consider $h\in \mathbb{Z}^{n+3}$ modulo the columns of $A$. It is easy to see that $\{Ae_1,\ldots,Ae_n\}\cup \{f_{\rho} :\rho=v_1,u_1,z_1 \}$ is a $\mathbb{Z}$-basis of $\mathbb{Z}^{n+3}$. Therefore, we will study $h$ of the following form
\[
h_{\rho}=\begin{cases}
	d, & \text{if $\rho=v_1$}\\
	f, & \text{if $\rho=u_1$}\\
	e, & \text{if $\rho=z_1$}\\
	0, & \text{otherwise}
\end{cases}
\]
for any $d,e,f\in \mathbb{Z}$.

The next step is to determine when $\varphi_h$ is convex. Using Theorem \ref{theoremconevex} we obtain that $\varphi_h$ is convex if and only if $\varphi_h$ satisfies the following inequalities
\begin{equation*}
	\begin{split}
		\varphi_h(v_1+\cdots+v_{p_0}+y_1+\cdots+y_{p_1})&\geq \varphi_h(v_1)+\cdots+ \varphi_h(v_{p_0})+\varphi_h(y_1)+\cdots+
		\varphi_h(y_{p_1}),\\
		\varphi_h( y_1+\cdots+y_{p_1}+z_1+\cdots+z_{p_2})&\geq \varphi_h(y_1)+\cdots+ \varphi_h(y_{p_1})+\varphi_h(z_1)+\cdots+
		\varphi_h(z_{p_2}),\\
		\varphi_h(z_1+\cdots+z_{p_2}+t_1+\cdots+t_{p_3})&\geq \varphi_h(z_1)+\cdots+ \varphi_h(z_{p_2})+\varphi_h(t_1)+\cdots+
		\varphi_h(t_{p_3}),\\
		\varphi_h(t_1+\cdots+t_{p_3}+u_1+\cdots+u_{p_4})&\geq \varphi_h(t_1)+\cdots+ \varphi_h(t_{p_3})+\varphi_h(u_1)+\cdots+
		\varphi_h(u_{p_4}),\\
		\varphi_h(u_1+\cdots+u_{p_4}+v_1+\cdots +v_{p_0})&\geq \varphi_h(u_1)+\cdots+ \varphi_h(u_{p_5})+\varphi_h(v_1)+\cdots+
		\varphi_h(v_{p_0}).
	\end{split}
\end{equation*}

Recall the construction of $\varphi_h$ from (\ref{support}). We use the associated primitive relations in Theorem \ref{Batyrev} to obtain a more concrete characterization. It follows that $\varphi_h$ is convex if and only if 
\begin{equation*}
	\begin{split}
		0\geq -d,\\
		-f\geq -e,\\
		0\geq -e,\\
		0\geq -f,\\
		0\geq -f -d.
	\end{split}    
\end{equation*}
That is $d,e,f\geq0$ and $e\geq f$. Hence, we can consider $h$ of the following form
\begin{equation}\label{height}
	h_{\rho}=\begin{cases}
		d, & \text{if $\rho=v_1$}\\
		f, & \text{if $\rho=u_1$}\\
		e+f, & \text{if $\rho=z_1$}\\
		0, & \text{otherwise}
	\end{cases}
\end{equation}
for any $d,e,f\in \mathbb{Z}_{\geq0}$.

In the following paragraphs, using (\ref{matrix}), we identify projections onto dilations of the standard simplex where the IDP is well-known. We can next verify the IDP for a pair of polytopes by carefully examining the fibers of these projections. First, we will discuss some additional notations.

Let $J,K,S,T$ be index sets given by
\begin{align*}
	J&=\{t_1,\ldots,t_{p_3},z_2,\ldots,z_{p_2}\},\\
	K&=J\cup\{z_1\},\\
	S&=\{v_1,\ldots,v_{p_0},u_2,\ldots,u_{p_4},y_2,\ldots,y_{p_1}\},\\
	T&=S\cup\{u_1,y_1\}.
\end{align*}
Now consider the projection maps to respective coordinates given by the index sets as follows:
\[
\begin{tikzcd}[column sep=small]
	& \mathbb{R}^{n+3} \arrow[dl,"\pi_T"] \arrow[dr,"\pi_K"] & &&&\mathbb{R}^{n} \arrow[dl,"\pi_S"] \arrow[dr,"\pi_J"]&\\
	\mathbb{R}^{|T|}  & & \mathbb{R}^{|K|}&& \mathbb{R}^{|S|}  & & \mathbb{R}^{|J|}
\end{tikzcd}
\]

Given a lattice polytope $P=P(A,h)$, we can write it as
\begin{equation*}
	P = \Bigg\{ (x_{S},x_{J})=x \in \mathbb{R}^n: \left[   
	\begin{array}{c|c}
		F   & G   \\
		\hline
		0   & H 
	\end{array}
	\right] . 
	\left[   
	\begin{array}{c}
		x_S   \\
		\hline
		x_J 
	\end{array}
	\right]
	\geq 
	\left[   
	\begin{array}{c}
		-h_T   \\
		\hline
		-h_K 
	\end{array}
	\right]
	\Bigg\}
\end{equation*}
where $x_S=\pi_S(x)$, $x_J=\pi_J(x)$ and $h_T=\pi_T(h)$, $h_K=\pi_K(h)$. First, consider the polytope $ P^{\Delta}$ defined by
\[  P^{\Delta}:=P(H,h_K)\subseteq \mathbb{R}^{|J|}.\]
It is worth noting that $P^{\Delta}=(e+f)\Delta_{|J|}$, where $\Delta_{|J|}=\text{Conv}(0,e_{n+1-|J|},\ldots,e_n)$ is the standard $|J|$-simplex. We can also look at $P^{\Delta}$ from a geometrical viewpoint: indeed, we can show that $P^{\Delta}=\pi_{J}(P)$. It is easy to see that $\pi_J(P)\subseteq P(H,h_K)$. Conversely, if $x_J \in P(H,h_K)$, then $(0,x_J)\in P$. 

Let $\alpha \in P \cap \mathbb{Z}^n$. Set $\theta(\alpha)=\pi_T(h)+G\alpha_J$ and define a new polytope
\[\widetilde{P}(\alpha):=P(F,\theta(\alpha)) \subseteq \mathbb{R}^{|S|}.\]
By construction, if $x_S \in \widetilde{P}(\alpha)$, then $(x_S,\alpha_J) \in P$. Hence, we can also write 
\[\widetilde{P}=\pi_S(\pi_J^{-1}(\alpha_J)\cap P).\]

Geometrically, $\widetilde{P}(\alpha)$ is constructed by intersecting the original polytope $P$ with hyperplane sections defined by $x_j=\alpha_j$ for $j\in J$ and projecting into $\mathbb{R}^{|S|}$. See Figure \ref{projection} for the three-dimensional interpretation.

\begin{figure}[htpb]
	\centering
	\begin{subfigure}[b]{0.2\textwidth}
		\begin{tikzpicture}[scale=0.3]
			\coordinate (A1) at (0,0);
			\coordinate (A2) at (6,0);
			\coordinate (A3) at (4.5,3);
			
			\coordinate (B1) at (0,11);
			\coordinate (B2) at (4,11);
			\coordinate (B3) at (3,13);
			
			\coordinate (C1) at (0,6);
			\coordinate (C2) at (8,6);
			\coordinate (C3) at (6,10);
			
			\coordinate (D1) at (9,6);
			\coordinate (D2) at (10,6);
			
			\coordinate (O) at (0,15);
			
			\coordinate (a1) at (-2,0);
			\coordinate (b1) at (-2,11);
			\coordinate (c1) at (-2,6);
			\begin{scope}[thick,opacity=0.6]
				\draw (A1) -- (A2);
				\draw (A1) -- (B1); 
				\draw (A2) -- (C2);
				\draw (C2) -- (B2);
				\draw (B1) -- (B2) -- (B3)--cycle;
				\draw  (C3)--(B3);
			\end{scope}
			
			\begin{scope}[thick,dashed,,opacity=0.6]
				\draw (A1) -- (A2) -- (A3)--cycle;
				\draw (B1) -- (B2) -- (B3)--cycle;
				\draw (C1) -- (C2) -- (C3)--cycle;
				\draw (A1) -- (C1);
				\draw (B1) -- (C1);
				\draw (A2) -- (C2);
				\draw (B2) -- (C2);
				\draw (A3) -- (C3);
				\draw (B3) -- (C3);
			\end{scope}
			\draw[fill=cof,opacity=0.6] (A1) -- (A2) -- (A3);
			\draw[fill=pur,opacity=0.6] (C1) -- (C2) -- (C3);
			\draw[fill=greeo,opacity=0.6] (B1) -- (B2) -- (B3);
			\filldraw[black] (A1) circle (5pt); 
			\filldraw[black] (A2) circle (5pt);
			\filldraw[black] (A3) circle (5pt);
			\filldraw[black] (B1) circle (5pt);
			\filldraw[black] (B2) circle (5pt);
			\filldraw[black] (B3) circle (5pt);
			\filldraw[black] (C2) circle (5pt);
			\filldraw[black] (C3) circle (5pt);
			\node at (O) {$P$};
		\end{tikzpicture}
	\end{subfigure}
	\quad
	\begin{subfigure}[b]{0.2\textwidth}
		\begin{tikzpicture}[scale=0.3]
			\draw (A1) -- (B1);
			\node at (O) {$P^{\Delta}$};
			\filldraw[black](A1) circle (2pt); 
			\node at (a1) {0};
			\filldraw[black](C1) circle (2pt);
			\node at (c1) {$f$};
			\filldraw[black](B1) circle (2pt);
			\node at (b1) {$f+e$};
		\end{tikzpicture}
	\end{subfigure}
	\begin{subfigure}[b]{0.2\textwidth}
		\begin{tikzpicture}[scale=0.3]
			\node at (0,16) {$\widetilde{P}$ };
			\draw (A1) -- (A2) -- (A3)--cycle;
			\draw[fill=cof,opacity=0.6] (A1) -- (A2) -- (A3);
			\draw[fill=pur,opacity=0.6] (C1) -- (C2) -- (C3);
			\draw[fill=greeo,opacity=0.6] (B1) -- (B2) -- (B3);
		\end{tikzpicture}
	\end{subfigure}
	\caption{Projection of a polytope}
	\label{projection}
\end{figure}
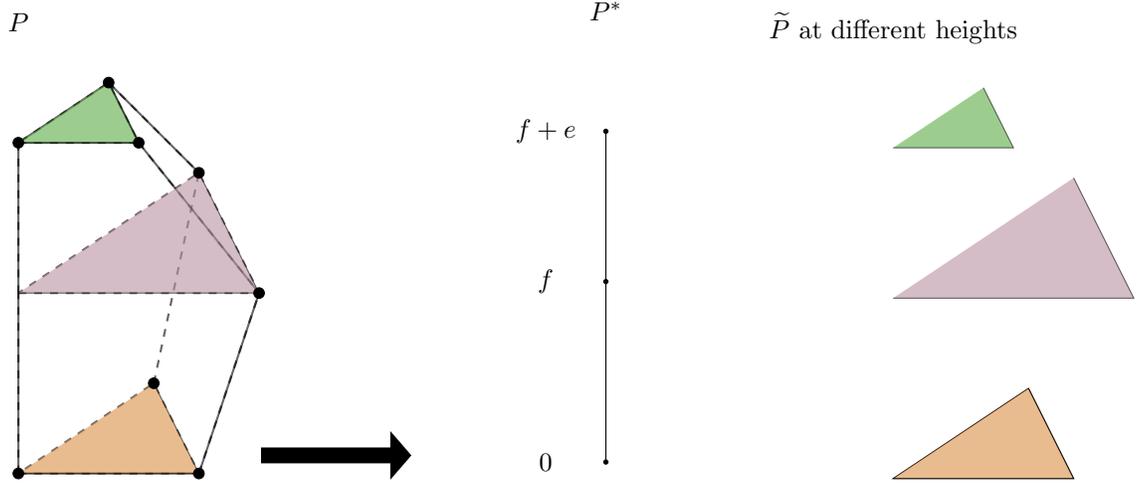

Using the construction of the above polytopes we will now use the following algorithm.

\begin{itemize}
	\item Given two lattice polytopes $P=P(A,h),Q=P(A,h')$  and their Minkowski sum $(P+Q)=P(A,h+h')$.
	\item We find projection to $\mathbb{R}^{|J|}$ and construct $P^{\Delta},Q^{\Delta}, (P+Q)^{\Delta}$. Note that the pair $(P^{\Delta},Q^{\Delta})$ has the IDP. 
	\item  For every $\alpha_J \in (P+Q) \cap \mathbb{Z}^{|J|}$ find all possible $\beta_J \in P \cap \mathbb{Z}^{|J|}$ and $\gamma_J \in Q \cap \mathbb{Z}^{|J|}$ such that $\beta_J +\gamma_J =\alpha_J$.
	\item Analyze the fibers of $\alpha_J,\beta_J$ and $\gamma_J$ by using the construction $\widetilde{P}(\beta_J),\widetilde{Q}(\gamma_J)$ and $\widetilde{(P+Q)}(\alpha_J)$.
	\item For all  $(\alpha_S,\alpha_J) \in \pi_J^{-1}(\alpha_J) \cap(P+Q)\cap \mathbb{Z}^n$ find $(\beta_J,\beta_S) \in \pi_J^{-1}(\beta_J) \cap P\cap \mathbb{Z}^n$ and $(\gamma_J,\gamma_S) \in \pi_J^{-1}(\gamma_J) \cap Q\cap \mathbb{Z}^n$ such that $(\beta_S,\beta_J)+(\gamma_s,\gamma_J)=(\alpha_S,\alpha_J)$.
\end{itemize}

With the following example, we will illustrate how to use the construction of $P^{\Delta}$ and $\widetilde{P}$ to prove the IDP.

\begin{example}\label{example8}
	Let  
	
	\begin{equation*}
		\begin{split}
			A=
			\begin{pNiceMatrix}[first-col]
				v_1& 1&0&0\\
				v_2 &0&1 &0   \\
				u_1&-1&-1 &1   \\
				y_1& -1 &-1 &2\\
				t_1&0 &0&1 \\
				z_1&0&0&-1
			\end{pNiceMatrix},
		\end{split}
		\quad
		\begin{split}
			h=
			\begin{pmatrix}
				0\\
				0\\
				3\\
				0\\
				0\\
				4\\
			\end{pmatrix},
		\end{split}
		\quad
		\begin{split}
			h'=
			\begin{pmatrix}
				2\\
				0\\
				1\\
				0\\
				0\\
				3
			\end{pmatrix},
		\end{split}
	\end{equation*}
	and $P=P(A,h)$, $Q=P(A,h')$. Then $P+Q=P(A,h+h')$. 
	We will show that $P\cap \mathbb{Z}^3 + Q\cap \mathbb{Z}^3 = (P+Q)\cap \mathbb{Z}^3$. In Figure \ref{secondproj}, we have drawn the polytopes $P^{\Delta},Q^{\Delta}$ and $(P+Q)^{\Delta}$ and it is easy to see that $P^{\Delta}+Q^{\Delta}=(P+Q)^{\Delta}$ and $(P^{\Delta},Q^{\Delta})$ has the IDP. Hence, for every $\alpha_{t_1}\in (P+Q)^{\Delta}\cap \mathbb{Z}$, we can find
	$\beta_{t_1} \in P^{\Delta}\cap \mathbb{Z},\gamma_{t_1} \in Q^{\Delta}\cap \mathbb{Z}$ such that $\beta_{t_1}+\gamma_{t_1}=\alpha_{t_1}$.
	
	Let 
	\begin{equation*}
		\begin{split}
			F=
			\begin{pNiceMatrix}[first-col]
				v_1& 1&0\\
				v_2 &0&1\\
				u_1&-1&-1\\
				y_1& -1 &-1
			\end{pNiceMatrix},    
		\end{split}
		\quad
		\begin{split}
			\theta(\beta_{t_1})=
			\begin{pmatrix}
				0\\
				0\\
				3+\beta_{t_1}\\
				2\beta_{t_1}\\
			\end{pmatrix},    
		\end{split}
		\quad
		\begin{split}
			\theta(\gamma_{t_1})=
			\begin{pmatrix}
				2\\
				0\\   
				1+\gamma_{t_1}\\
				2\gamma_{t_1}\\
			\end{pmatrix}.    
		\end{split}
	\end{equation*}
	Then $\widetilde{P}(\beta_{t_1})=P(F,\theta(\beta_{t_1}))$, $\widetilde{Q}(\gamma_{t_1})=P(F,\theta(\gamma_{t_1}))$ and $\widetilde{(P+Q)}(\alpha_{t_1})=P(F,\theta(\alpha_{t_1}))$. Let $(\alpha_{v_1},\alpha_{v_2},\alpha_{t_1}) \in (P+Q)\cap \mathbb{Z}^3$.
	The basic idea is to show that for every $\alpha_{t_1}\in (P+Q)^{\Delta}\cap \mathbb{Z}$, we can find
	$\beta_{t_1},\gamma_{t_1}$ additionally satisfying the following two conditions:
	\begin{enumerate}
		\item $\widetilde{P}(\beta_{t_1})+\widetilde{Q}(\gamma_{t_1})=\widetilde{(P+Q)}(\alpha_{t_1})$.
		\item $(\widetilde{P}(\beta_{t_1}),\widetilde{Q}(\gamma_{t_1}))$ has the IDP.
	\end{enumerate}
	Suppose such a choice exists. Since $(\alpha_{v_1},\alpha_{v_2}) \in \widetilde{(P+Q)}(\alpha_{t_1})$ we can find $(\beta_{v_1},\beta_{v_2})\in \widetilde{P}(\beta_{t_1})\cap \mathbb{Z}^2$ and $(\gamma_{v_1},\gamma_{v_2})\in \widetilde{P}(\gamma_{t_1})\cap \mathbb{Z}^2$ such that
	\[(\beta_{v_1},\beta_{v_2})+(\gamma_{v_1},\gamma_{v_2})=(\alpha_{v_1},\alpha_{v_2}).\]
	Then we conclude the proof in this case by noticing that 
	\[(\beta_{v_1},\beta_{v_2},\beta_{t_1})+(\gamma_{v_1},\gamma_{v_2},\gamma_{t_1})=(\alpha_{v_1},\alpha_{v_2},\alpha_{t_1}).\]

	Hence, it is enough to show that choices of
	$\beta_{t_1},\gamma_{t_1}$ additionally satisfying the extra two conditions. We proceed by two cases depending on $0\leq \alpha_{t_1}\leq4$ or $4\leq \alpha_{t_1}\leq7$. 
	
	\textbf{Case 1.} If $4\leq \alpha_{t_1} \leq 7$, then choose $3\leq \beta_{t_1} \leq 4,\beta_{t_1}\in \mathbb{Z}$ and $1\leq \gamma_{t_1}\leq 3,\gamma_{t_1} \in \mathbb{Z}$ such that $\alpha_{t_1}=\beta_{t_1}+\gamma_{t_1}$.  The case $\alpha_{t_1}=7$, is shown in Figure \ref{projectionfirst}. Note that the half-space generated by the ray $u_1$ is not a supporting half-space for $\widetilde{(P+Q)}(\alpha_{t_1})$. 
	Our choice of $\beta_{t_1},\gamma_{t_1}$ ensure that the half-space generated by the ray $u_1$ is also not a supporting half-space for $\widetilde{P}(\beta_{t_1})$ and $\widetilde{Q}(\gamma_{t_1})$. Hence, we have $\widetilde{P}=P(\widetilde{F},\widetilde{\theta})$, $\widetilde{Q}=P(\widetilde{F},\widetilde{\theta'})$, where
	
	\begin{equation*}
		\begin{split}
			\widetilde{F}=
			\begin{pNiceMatrix}[first-col]
				v_1& 1&0\\
				v_2 &0&1    \\
				y_1& -1 &-1
			\end{pNiceMatrix},    
		\end{split}
		\quad
		\begin{split}
			\widetilde{\theta}=
			\begin{pmatrix}
				0\\
				0\\
				2\beta_{t_1}
			\end{pmatrix},    
		\end{split}
		\quad
		\begin{split}
			\widetilde{\theta'}=
			\begin{pmatrix}
				2\\
				0\\   
				2\gamma_{t_1}
			\end{pmatrix}.   
		\end{split}
	\end{equation*}
	Since $\widetilde{F}$ corresponds to a smooth complete splitting fan in $\mathbb{R}^2$ and $\widetilde{\theta},\widetilde{\theta'}$ are convex support functions we have $\widetilde{P}+\widetilde{Q}=\widetilde{P+Q}$ and $(\widetilde{P},\widetilde{Q})$ has IDP.
	
	\textbf{Case 2.} If $0\leq \alpha_{t_1} \leq 4$, then choose $0\leq \beta_{t_1} \leq 3,\beta_{t_1}\in \mathbb{Z}$ and $0\leq \gamma_{t_1}\leq 1,\gamma_{t_1} \in \mathbb{Z}$ such that $\alpha_{t_1}=\beta_{t_1}+\gamma_{t_1}$.  The case $\alpha_{t_1}=4$, is shown in Figure \ref{projectiofirst2} and the case $\alpha_{t_1}=0$, is shown in Figure \ref{projectiofirst3}. Note that the half-space generated by the ray $y_1$ is not a supporting half-space for $\widetilde{(P+Q)}(\alpha_{t_1})$. 
	Our choice of $\beta_{t_1},\gamma_{t_1}$ ensure that the half-space generated by the ray $y_1$ is also not a supporting half-space for $\widetilde{P}(\beta_{t_1})$ and $\widetilde{Q}(\gamma_{t_1})$. Hence, we have $\widetilde{P}=P(\widetilde{F},\widetilde{\theta})$, $\widetilde{Q}=P(\widetilde{F},\widetilde{\theta'})$, where
	\begin{equation*}
		\begin{split}
			\widetilde{F}=
			\begin{pNiceMatrix}[first-col]
				v_1& 1&0\\
				v_2 &0&1    \\
				u_1& -1 &-1
			\end{pNiceMatrix},    
		\end{split}
		\quad
		\begin{split}
			\widetilde{\theta}=
			\begin{pmatrix}
				0\\
				0\\
				\beta_{t_1}+3\\
			\end{pmatrix},    
		\end{split}
		\quad
		\begin{split}
			\widetilde{\theta'}=
			\begin{pmatrix}
				2\\
				0\\   
				\gamma_{t_1}+1\\
			\end{pmatrix}.    
		\end{split}   
	\end{equation*}
	Since $\widetilde{F}$ corresponds to a smooth complete splitting fan in $\mathbb{R}^2$ and $\widetilde{\theta},\widetilde{\theta'}$ are convex support functions we have $\widetilde{P}+\widetilde{Q}=\widetilde{P+Q}$ and $(\widetilde{P},\widetilde{Q})$ has IDP.
\end{example}

\begin{remark}
	If we use a similar type of construction for a general lattice polytope, $\widetilde{P}(\alpha)$ does not have to be a lattice polytope. However, in our case, we can associate it with a smooth complete fan and a convex support function (see cases 1,2,3,4 in the Proof of Theorem 1.4). As a result, it is always a lattice polytope.
\end{remark}

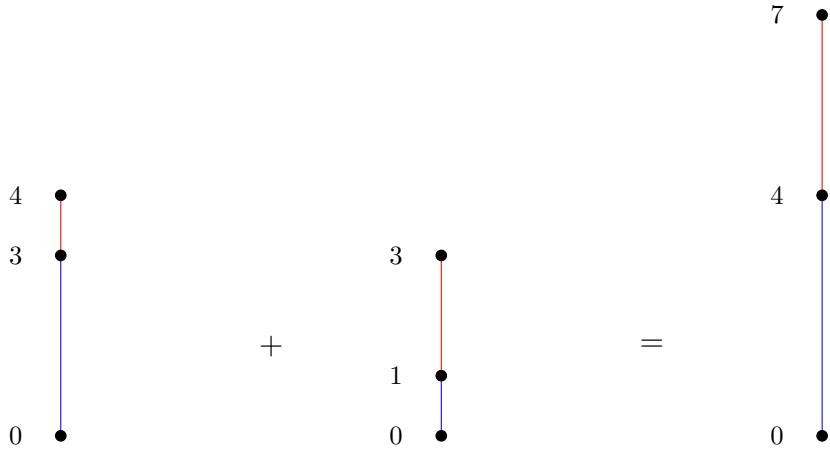
\begin{figure}[htpb]
	\centering
	\begin{subfigure}[b]{0.2\textwidth}
		\begin{tikzpicture}[scale=0.2]
			
			\coordinate (O) at (0,-5);
			
			\coordinate (A11) at (0,0); 
			\coordinate (B11) at (0,12); 
			\coordinate (C11) at (0,16);
			
			\draw[blue] (A11)--(B11);
			\draw[red] (B11)--(C11);
			
			\filldraw[black] (A11) circle (10pt);
			\filldraw[black] (B11) circle (10pt);
			\filldraw[black] (C11) circle (10pt);
			
			\node at ($ (A11) + (-3,0) $) {$0$};
			\node at ($ (B11) + (-3,0) $) {$3$};
			\node at ($ (C11) + (-3,0) $) {$4$};
			\node at (10,6) {\textbf{+}};
		\end{tikzpicture}
	\end{subfigure}
	\quad 
	\begin{subfigure}[b]{0.2\textwidth}
		\begin{tikzpicture}[scale=0.2]
			\coordinate (A11) at (0,0); 
			\coordinate (B11) at (0,4); 
			\coordinate (C11) at (0,12);
			\draw[blue] (A11)--(B11);
			\draw[red] (B11)--(C11);
			
			\filldraw[black] (A11) circle (10pt);
			\filldraw[black] (B11) circle (10pt);
			\filldraw[black] (C11) circle (10pt);
			
			\node at ($ (A11) + (-3,0) $) {$0$};
			\node at ($ (B11) + (-3,0) $) {$1$};
			\node at ($ (C11) + (-3,0) $) {$3$};
			\node at (10,6) {\textbf{=}};
		\end{tikzpicture}
	\end{subfigure}
	\quad
	\begin{subfigure}[b]{0.2\textwidth}
		\begin{tikzpicture}[scale=0.2]

			\coordinate (A11) at (0,0); 
			\coordinate (B11) at (0,16); 
			\coordinate (C11) at (0,28);
			\draw[blue] (A11)--(B11);
			\draw[red] (B11)--(C11);
			
			\filldraw[black] (A11) circle (10pt);
			\filldraw[black] (B11) circle (10pt);
			\filldraw[black] (C11) circle (10pt);
			
			\node at ($ (A11) + (-3,0) $) {$0$};
			\node at ($ (B11) + (-3,0) $) {$4$};
			\node at ($ (C11) + (-3,0) $) {$7$};
			
		\end{tikzpicture}
	\end{subfigure}
	\caption{Polytopes $P^{\Delta},Q^{\Delta},(P+Q)^{\Delta}$}
	\label{secondproj}
\end{figure}

\begin{figure}[htpb]
	\centering
	\begin{subfigure}[b]{0.2\textwidth}
		\begin{tikzpicture}[scale=0.15]
			\coordinate (C11) at (0,0);
			\coordinate (C12) at (7,0);
			\coordinate (C13) at (0,7);
			
			\draw (C11)--(C12)--(C13)--cycle;
			
			\node at (11,3) {\textbf{+}};
			
			\filldraw[black] (C11) circle (10pt);
			\node at ($(C11)+(-2,-2)$) {$(0,0)$};
			
			\filldraw[black] (C12) circle (10pt);
			\node at ($(C12)+(2,-2)$) {$(7,0)$};
			
			\filldraw[black] (C13) circle (10pt);
			\node at ($(C13)+(0,2)$) {$(0,7)$};
			
		\end{tikzpicture}
	\end{subfigure}\quad
	\begin{subfigure}[b]{0.2\textwidth}
		\begin{tikzpicture}[scale=0.15]
			\coordinate (C11) at (0,0);
			\coordinate (C12) at (6,0);
			\coordinate (C13) at (0,6);
			\draw (C11)--(C12)--(C13)--cycle;
			
			\node at (10,3) {\textbf{=}};
			
			\filldraw[black] (C11) circle (10pt);
			\node at ($(C11)+(-2,-2)$) {$(-2,0)$};
			
			\filldraw[black] (C12) circle (10pt);
			\node at ($(C12)+(2,-2)$) {$(4,0)$};
			
			\filldraw[black] (C13) circle (10pt);
			\node at ($(C13)+(0,2)$) {$(-2,6)$};
			
		\end{tikzpicture}
	\end{subfigure}\quad
	\begin{subfigure}[b]{0.2\textwidth}
		\begin{tikzpicture}[scale=0.15]
			\coordinate (C11) at (-1,0);
			\coordinate (C12) at (12,0);
			\coordinate (C13) at (-1,13);
			
			\draw (C11)--(C12)--(C13)--cycle;
			
			\filldraw[black] (C11) circle (10pt);
			\node at ($(C11)+(-2,-2)$) {$(-2,0)$};
			
			\filldraw[black] (C12) circle (10pt);
			\node at ($(7,-2)$) {(11,0)};
			\filldraw[black] (C13) circle (10pt);
			\node at ($(C13)+(0,2)$) {$(-2,13)$};
		\end{tikzpicture}
	\end{subfigure}
	\caption{$\widetilde{P(4)},\widetilde{Q(3)},\widetilde{(P+Q)(7)}$}
	\label{projectionfirst}
\end{figure}
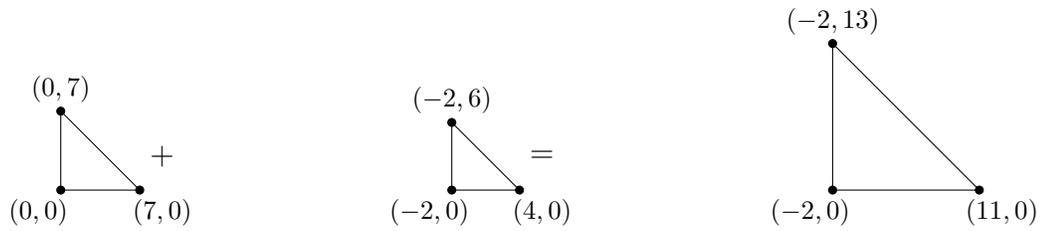

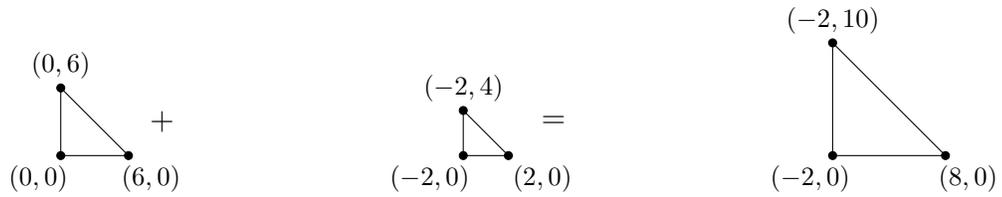
\begin{figure}[htpb]
	\centering
	\begin{subfigure}[b]{0.2\textwidth}
		\begin{tikzpicture}[scale=0.15]
			
			\coordinate (B11) at (0,0); 
			\coordinate (B12) at (6,0);
			\coordinate (B13) at (0,6);

			\draw (B11)--(B12)--(B13)--cycle;
			
			\filldraw[black] (B11) circle (10pt);
			\node at ($(B11)+(-2,-2)$) {$(0,0)$};
			
			\filldraw[black] (B12) circle (10pt);
			\node at ($(B12)+(2,-2)$) {$(6,0)$};
			
			\filldraw[black] (B13) circle (10pt);
			\node at ($(B13)+(0,2)$) {$(0,6)$};
			
			\node at (9,3) {+};
			
		\end{tikzpicture}
	\end{subfigure}\quad
	\begin{subfigure}[b]{0.2\textwidth}
		\begin{tikzpicture}[scale=0.15]
			
			\coordinate (B11) at (0,0); 
			\coordinate (B12) at (4,0);
			\coordinate (B13) at (0,4);
			
			\draw (B11)--(B12)--(B13)--cycle;
			
			\filldraw[black] (B11) circle (10pt);
			\node at ($(B11)+(-3,-2)$) {$(-2,0)$};
			
			\filldraw[black] (B12) circle (10pt);
			\node at ($(B12)+(3,-2)$) {$(2,0)$};
			
			\filldraw[black] (B13) circle (10pt);
			\node at ($(B13)+(0,2)$) {$(-2,4)$};
			
			\node at (8,2) {=};
			
		\end{tikzpicture}
	\end{subfigure}\quad
	\begin{subfigure}[b]{0.2\textwidth}
		\begin{tikzpicture}[scale=0.15]
			\coordinate (B11) at (0,0); 
			\coordinate (B12) at (10,0);
			\coordinate (B13) at (0,10);
			
			\draw (B11)--(B12)--(B13)--cycle;
			
			\filldraw[black] (B11) circle (10pt);
			\node at ($(B11)+(-2,-2)$) {$(-2,0)$};
			
			\filldraw[black] (B12) circle (10pt);
			\node at ($(9,-2)$) {$(8,0)$};
			
			\filldraw[black] (B13) circle (10pt);
			\node at ($(B13)+(0,2)$) {$(-2,10)$};
		\end{tikzpicture}
	\end{subfigure}
	\caption{$\widetilde{P(3)},\widetilde{Q(1)},\widetilde{(P+Q)(4)}$}
	\label{projectiofirst2}
\end{figure}

\begin{figure}[htpb]
	\centering
	\begin{subfigure}[b]{0.2\textwidth}
		\begin{tikzpicture}[scale=0.2]
			
			\coordinate (A11) at (0,0); 
			
			\filldraw[black] (A11) circle (10pt);
			\node at ($(A11)+(-2,-2)$) {$(0,0)$};

			\node at (9,2) {+};
			
		\end{tikzpicture}
	\end{subfigure}\quad
	\begin{subfigure}[b]{0.2\textwidth}
		\begin{tikzpicture}[scale=0.2]
			
			\coordinate (A11) at (0,0); 
			\coordinate (A12) at (2,0); 
			\coordinate (A13) at (0,2); 
			
			\draw (A11)--(A12)--(A13)--cycle;
			
			\filldraw[black] (A11) circle (10pt);
			\node at ($(A11)+(-3,-2)$) {$(-2,0)$};
			
			\filldraw[black] (A12) circle (10pt);
			\node at ($(A12)+(3,-2)$) {$(0,0)$};
			
			\filldraw[black] (A13) circle (10pt);
			\node at ($(A13)+(0,2)$) {$(-2,2)$};
			\node at (6,2) {=};
		\end{tikzpicture}
	\end{subfigure}\quad
	\begin{subfigure}[b]{0.2\textwidth}
		\begin{tikzpicture}[scale=0.2]
			
			\coordinate (A11) at (0,0); 
			\coordinate (A12) at (2,0); 
			\coordinate (A13) at (0,2); 
			
			\draw (A11)--(A12)--(A13)--cycle;
			
			\filldraw[black] (A11) circle (10pt);
			\node at ($(A11)+(-3,-2)$) {$(-2,0)$};
			
			\filldraw[black] (A12) circle (10pt);
			\node at ($(A12)+(3,-2)$) {$(0,0)$};
			
			\filldraw[black] (A13) circle (10pt);
			\node at ($(A13)+(0,2)$) {$(-2,2)$};
			
		\end{tikzpicture}
	\end{subfigure}
	\caption{$\widetilde{P(0)},\widetilde{Q(0)},\widetilde{(P+Q)(0)}$}
	\label{projectiofirst3}
\end{figure}
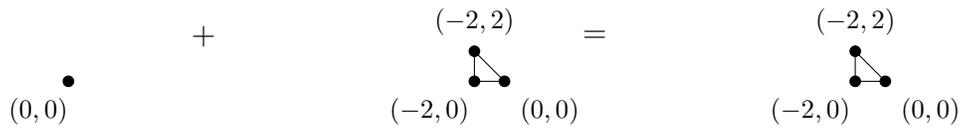

\begin{proof}[Proof of Theorem \ref{IDPtheorem}]
	Let $\Sigma$ be a smooth complete fan in $\mathbb{R}^n$. Suppose $\Sigma$ has at most $n+2$ rays or a splitting fan with $n+3$ rays. Then the result follows from Theorem \ref{IDP4}. Hence we can assume $P=P(A,h)$ and $Q=P(A,h')$, where $A$ is a matrix as in \eqref{matrix} and $h,h'$ are convex support functions as in \eqref{height}. Then we have $P+Q=P(A,h+h')$ by \eqref{sumpolytopes}. Our goal is to show that the pair $(P,Q)$ has the IDP. Let $\alpha \in (P+Q)\cap \mathbb{Z}^n$. We will find $\beta \in P\cap \mathbb{Z}^n$ and $\gamma \in Q\cap \mathbb{Z}^n$ such that $\beta+\gamma=\alpha$. We will use the projection technique discussed in the beginning of this section. 
	
	Let
	\begin{align*}
		P^{\Delta}&=P(H,h_{K}), & Q^{\Delta}&=P(H,h'_{K}), & (P+Q)^{\Delta}&=P(H,(h+h')_{K}).
	\end{align*}
	Since $P^{\Delta}=(e+f)\Delta_{|J|},Q^{\Delta}=(e'+f')\Delta_{|J|}$ and $(P+Q)^{\Delta}=(e+e'+f+f')\Delta_{|J|}$ we have 
	\[(P^{\Delta}\cap \mathbb{Z}^{|J|}) + (Q^{\Delta}\cap \mathbb{Z}^{|J|}) = (P+Q)^{\Delta}\cap \mathbb{Z}^{|J|}.\]
	Also notice that $\alpha_{J} \in (P+Q)^{\Delta}\cap \mathbb{Z}^{|J|}$. Hence, we can find $\beta_{J} \in P^{\Delta}\cap \mathbb{Z}^{|J|}$ and $\gamma_{J} \in Q^{\Delta}\cap \mathbb{Z}^{|J|}$ such that $\alpha_J=\beta_{J}+\gamma_J$. Furthermore, we will consider two subcases depending on the value of $\sum_{i\in X_3}\alpha_i$.

	\begin{itemize}
		\item If  $0\leq \sum_{i\in X_3} \alpha_i\leq f+f'$, we can choose $\beta_{J},\gamma_J$ such that $ 0\leq \sum_{i\in X_3} \beta_i\leq f$ and $0\leq \sum_{i\in X_3} \gamma_i\leq f'$.
		\item Similarly, if $ f+f'\leq \sum_{i\in X_3}\alpha_i\leq f+f'+e+e'$, we can choose  $\beta_{J},\gamma_J$ such that $ f\leq \sum_{i\in X_3} \beta_i\leq f+e$ and $f'\leq \sum_{i\in X_3} \gamma_i\leq f'+e'$.  
	\end{itemize}
	Now consider polytopes
	\begin{align*}
		\widetilde{P}&=P(F,\theta), & \widetilde{Q}&=P(F,\theta'), &\widetilde{P+Q}&=P(F,(\theta+\theta')).
	\end{align*}
	Note that $F$ is an $(|S|+2)\times |S|$ matrix. Depending on the values of $p_1$ and $p_4$, we will consider four cases. In each case, we will consider two subcases depending on the value of $\sum_{i\in X_3}\alpha_j$. Further to that, in every subcase, we can find an alternative description of $\widetilde{P},\widetilde{Q}$ and $\widetilde{P+Q}$ as

	\begin{align*}
		\widetilde{P}&=P(\widetilde{F},\widetilde{\theta}), & \widetilde{Q}&=P(\widetilde{F},\widetilde{\theta'}), &\widetilde{P+Q}&=P(\widetilde{F},\widetilde{\theta}+\widetilde{\theta'}),
	\end{align*}
	where $\widetilde{F}$ is associated to a smooth complete splitting fan $\widetilde{\Sigma}$ and $\widetilde{\theta},\widetilde{\theta'}$ are convex support functions for $\widetilde{\Sigma}$. Then by Theorem \ref{IDP4} we have 
	\[\widetilde{P}\cap \mathbb{Z}^{|S|}+\widetilde{Q}\cap \mathbb{Z}^{|S|}=\widetilde{P+Q}\cap \mathbb{Z}^{|S|}.\] 
	Hence, we can find $\beta_{L} \in \widetilde{P}\cap \mathbb{Z}^{|S|}$ and $\gamma_{L} \in \widetilde{Q}\cap \mathbb{Z}^{|S|}$ such that $\alpha_L=\beta_{L}+\gamma_L$.
	Finally, we can complete the proof by noticing that $(\beta_L,\beta_J)\in P\cap \mathbb{Z}^n$, $(\gamma_L,\gamma_J)\in Q\cap \mathbb{Z}^n$ and $(\beta_L,\beta_J)+(\gamma_L,\gamma_J)=(\alpha_L,\alpha_J)$.
	
	Here are the four cases.
	
	\textbf{Case 1.}
	If $p_1=p_4=1$, then we have 
	
	\begin{equation*}
		\begin{split}
			F=
			\begin{pNiceMatrix}[first-row,first-col,margin,hvlines]
				& & \Ldots[line-style={solid,<->},shorten=0pt]^{p_0\text{ columns}} \\
				{v_1}& \Block{3-3}<\LARGE>{\blue I} &&\\
				\vdots &&  &   \\
				{v_{p_0}}&& &   \\
				{u_1}& \blue -1 & \blue \Cdots&\blue -1\\
				{y_1}&\blue -1 &\blue \Cdots&\blue -1 \\
			\end{pNiceMatrix},
		\end{split}
		\quad
		\begin{split}
			\theta+\theta'=
			\begin{pNiceMatrix}[first-col]
				v_1& d\\
				\vdots&\vdots\\
				v_{p_0} &0 \\
				u_1& f+\sum_{i\in X_3} b_i\alpha_i+\sum_{j\in X_2\backslash\{z_1\}} c_j\alpha_j \\
				y_1& \sum_{i\in X_3} (b_i+1)\alpha_i+\sum_{j\in X_2\backslash\{z_1\}} c_j\alpha_j \\
			\end{pNiceMatrix},    
		\end{split}
	\end{equation*}

	\begin{equation*}
		\begin{split}
			\theta=
			\begin{pNiceMatrix}[first-col]
				v_1& d\\
				\vdots&\vdots\\
				v_{p_0} &0 \\
				u_1&\displaystyle{f+\sum_{i\in X_3} b_i\beta_i+\sum_{j\in X_2\backslash\{z_1\}} c_j\beta_j} \\
				y_1& \displaystyle{\sum_{i\in X_3} (b_i+1)\beta_i+\sum_{j\in X_2\backslash\{z_1\}} c_j\beta_j} \\
			\end{pNiceMatrix},
		\end{split}
		\quad
		\begin{split}
			\theta'=
			\begin{pmatrix}
				d'\\
				\vdots\\
				0 \\
				\displaystyle{f'+\sum_{i\in X_3} b_i\gamma_i+\sum_{j\in X_2\backslash\{z_1\}} c_j\gamma_j} \\
				\displaystyle{\sum_{i\in X_3} (b_i+1)\gamma_i+\sum_{j\in X_2\backslash\{z_1\}} c_j\gamma_j} \\ 
			\end{pmatrix}.
		\end{split}
	\end{equation*}

	As discussed in the previous paragraph, we need to associate the above data to a smooth complete fan and convex support functions. We first analyze the polytope $\widetilde{(P+Q)}(\alpha)=P(F,\theta+\theta')$. 
	
	It is clear that depending on the sum $\sum_{i\in X_3}\alpha_i$ one of the equations from the rows $u_1,y_1$ is redundant. 
	\begin{itemize}
		\item If $0\leq \sum_{i\in X_3} \alpha_i\leq f+f'$, we choose $\beta_J,\gamma_J$ such that $ 0\leq \sum_{i\in X_3} \beta_i\leq f$ and $0\leq \sum_{i\in X_3} \gamma_i\leq f'$. In this case, we construct $\widetilde{F},\widetilde{\theta},\widetilde{\theta'},\widetilde{\theta+\theta'}$ from $F,\theta,\theta',\theta+\theta'$ by removing the row $u_1$. Then we have 
		\begin{align*}
			\widetilde{P}(\beta)&=P(\widetilde{F},\widetilde{\theta}), & \widetilde{Q}(\gamma)&=P(\widetilde{F},\widetilde{\theta'}), &\widetilde{(P+Q)}(\alpha)&=P(\widetilde{F},\widetilde{\theta+\theta'}).
		\end{align*}
		\item If $f+f'\leq \sum_{i\in X_3} \alpha_i\leq f+e+f'+e'$, we choose $\beta_J,\gamma_J$ such that $ f\leq \sum_{i\in X_3} \beta_i\leq f+e$ and $f'\leq \sum_{i\in X_3} \gamma_i\leq f'+e'$. In this case, we construct $\widetilde{F},\widetilde{\theta},\widetilde{\theta'},\widetilde{\theta+\theta'}$ from $F,\theta,\theta',\theta+\theta'$ by removing the row $y_1$. Then we have 
		\begin{align*}
			\widetilde{P}(\beta)&=P(\widetilde{F},\widetilde{\theta}), & \widetilde{Q}(\gamma)&=P(\widetilde{F},\widetilde{\theta'}), &\widetilde{(P+Q)}(\alpha)&=P(\widetilde{F},\widetilde{\theta+\theta'}).
		\end{align*}    
	\end{itemize}
	Note that in both cases $\widetilde{F}$ is the same. Consider the smooth complete splitting fan $\Sigma$ in $\mathbb{R}^{p_0}$ with $p_0+1$ rays (see Example \ref{rank1}). Then the matrix representing the ray generators of $\Sigma$ is equal to $\widetilde{F}$. Also note that entries in  $\widetilde{\theta},\widetilde{\theta'}$ are nonnegative. Hence, it follows that they define convex support functions on $\Sigma$ (see Example \ref{rank1}).
	
	\textbf{Case 2.}
	If $p_1,p_4\geq2$, then we have
	\setcounter{MaxMatrixCols}{20}
	\[
	F=
	\begin{pNiceMatrix}[first-row,first-col,margin,hvlines]
		& & \Ldots[line-style={solid,<->},shorten=0pt]^{p_0+p_1+p_4-2 \text{ columns}} \\
		v_1& \Block{3-3}<\LARGE>{\blue I} & & &\Block{3-3}<\LARGE>{ \blue 0}&&&\Block{3-3}<\LARGE>{\blue 0}&&\\
		\vdots &&  & &&&&&  \\
		v_{p_0}&& & &  \\
		u_1& \blue -1 & \blue \Cdots&\blue -1&\blue -1&\blue \Cdots&\blue -1&\blue 0&\blue \Cdots&\blue 0\\
		u_2&\Block{3-3}<\LARGE>{\blue 0} & & &\Block{3-3}<\LARGE>{\blue I}&&&\Block{3-3}<\LARGE>{\blue 0}&&\\
		\vdots &&  & &&&&&  \\
		u_{p_4}&& & &  \\
		y_1&\blue -1 &\blue \Cdots&\blue -1 &\blue 0&\blue \Cdots& \blue 0&\blue -1&\blue \Cdots&\blue -1\\
		y_2&\Block{3-3}<\LARGE>{\blue 0} & & &\Block{3-3}<\LARGE>{\blue 0}&&&\Block{3-3}<\LARGE>{\blue I}&&\\
		\vdots &&  & &&&&&  \\
		y_{p_1}&& & &  \\
	\end{pNiceMatrix}
	,
	\]

	\begin{equation*}
		\begin{split}
			\theta=
			\begin{pNiceMatrix}[first-col]
				v_1& d\\
				\vdots&\vdots\\
				v_{p_0} &0 \\
				u_1&\displaystyle{f+\sum_{i\in X_3} b_i\beta_i+\sum_{j\in X_2\backslash\{z_1\}} c_j\beta_j} \\
				u_2&0\\
				\vdots&\vdots\\
				u_{p_4}&0\\
				y_1& \displaystyle{\sum_{i\in X_3} (b_i+1)\beta_i+\sum_{j\in X_2\backslash\{z_1\}} c_j\beta_j} \\
				y_2&0\\
				\vdots&\vdots\\
				y_{p_1}&0\\
			\end{pNiceMatrix},
		\end{split}
		\quad
		\begin{split}
			\theta'=
			\begin{pmatrix}
				d' \\
				\vdots\\
				0\\
				\displaystyle{f'+\sum_{i\in X_3} b_i\gamma_i+\sum_{j\in X_2\backslash\{z_1\}} c_j\gamma_j} \\
				0\\
				\vdots\\
				0\\
				\displaystyle{\sum_{i\in X_3} (b_i+1)\gamma_i+\sum_{j\in X_2\backslash\{z_1\}} c_j\gamma_j} \\
				0\\
				\vdots\\
				0
			\end{pmatrix}.
		\end{split}
	\end{equation*}
	
	Consider two fans $\Sigma_1,\Sigma_2$ with the rays coming from the rows of of $F$ but with different primitive collections. For $\Sigma_1$, we consider the primitive collections given by $\{v_1,\ldots,v_{p_0},u_1,\ldots,u_{p_4}\}, \{y_1,\ldots,y_{p_1}\}$. The primitive relations are given by
	\begin{equation*}
		\begin{split}
			v_1+\cdots+v_{p_0}+u_1+\cdots+u_{p_4}&= 0\\
			y_1+\cdots+y_{p_1}&= u_1+\cdots +u_{p_4}.
		\end{split}
	\end{equation*}  
	
	The support function $\varphi_{\theta}$ is convex with respect to $\Sigma_1$ if and only if
	\begin{equation*}
		\begin{split}
			\varphi_{\theta}( v_1+\cdots+v_{p_0}+u_1+\cdots+u_{p_4}) &\geq \varphi_{\theta}(v_1)+\cdots+\varphi_{\theta}(v_{p_0})+\varphi_{\theta}(u_1)+\cdots+\varphi_{\theta}(u_{p_4})\\
			\varphi_{\theta}( y_1+\cdots+y_{p_1})&\geq \varphi_{\theta}(y_1)+\cdots+\varphi_{\theta}(y_{p_1}).
		\end{split}
	\end{equation*}
	Simplifying the first inequality gives
	\[ 0\geq -d -\Bigg(f+\sum_{i\in X_3} b_i\beta_i+\sum_{j\in X_2\backslash\{z_1\}} c_j\beta_j \Bigg),\]
	which holds in this case because $d,b_i,c_j,\beta_j$ are all non-negative. If we simplify the second inequality, we get 
	\[ -\Bigg(f+\sum_{i\in X_3} b_i\beta_i+\sum_{j\in X_2\backslash\{z_1\}} c_j\beta_j \Bigg)\geq -\Bigg(\sum_{i\in X_3} (b_i+1)\beta_i+\sum_{j\in X_2\backslash\{z_1\}} c_j\beta_j\Bigg)\]
	which holds if and only if $\sum_{i\in X_3}\beta_i\geq f$. Similarly, we get  $\varphi_{\theta'}$ is convex if and only if $\sum_{i\in X_3} \gamma_i\geq f' $ and $\varphi_{(\theta+\theta')}$ is convex if and only if $\sum_{i\in X_3} \alpha_i\geq f+f'$. As a result, if 
	\[f+f' \leq \sum_{i\in X_3} \alpha_i \leq f+e+f'+e',\] 
	\noindent
	we find  $\beta_J,\gamma_J \in \mathbb{Z}^{|J|} $ such that 
	\[f \leq \sum_{i\in X_3} \beta_i  \leq f+e, \quad f' \leq \sum_{i\in X_3} \gamma_i \leq f'+e'.\]
	We have seen that $\theta,\theta'$ are convex in this case.
	
	For the fan $\Sigma_2$, we consider the primitive collections given by $\{v_1,\ldots,v_{p_0},y_1,\ldots,y_{p_1}\}$, $\{u_1,\ldots,u_{p_4}\}$. The primitive relations are given by
	\begin{equation*}
		\begin{split}
			v_1+\cdots+v_{p_0}+ y_1+\cdots+y_{p_1}&=0\\
			u_1+\cdots+u_{p_4}&=y_1+\cdots+y_{p_1}.
		\end{split}    
	\end{equation*}
	
	The support function $\varphi_{\theta}$ is convex with respect to $\Sigma_1$ if and only if
	\begin{equation*}
		\begin{split}
			\varphi_{\theta}( v_1+\cdots+v_{p_0}+y_1+\cdots+y_{p_1}) &\geq \varphi_{\theta}(v_1)+\cdots+\varphi_{\theta}(v_{p_0})+\varphi_{\theta}(y_1)+\cdots+\varphi_{\theta}(y_{p_4})\\
			\varphi_{\theta}( u_1+\cdots+u_{p_4})&\geq \varphi_{\theta}(u_1)+\cdots+\varphi_{\theta}(u_{p_4}).
		\end{split}
	\end{equation*}
	Simplifying the first inequality gives
	\[ 0\geq -d -\Bigg(f+\sum_{i\in X_3} (b_i+1)\beta_i+\sum_{j\in X_2\backslash\{z_1\}} c_j\beta_j \Bigg),\]
	which holds in this case because $d,b_i,c_j,\beta_j$ are all non-negative. If we simplify the second inequality, we get 
	\[ -\Bigg(\sum_{i\in X_3} (b_i+1)\beta_i+\sum_{j\in X_2\backslash\{z_1\}} c_j\beta_j \Bigg)\geq -\Bigg(f+\sum_{i\in X_3} b_i\beta_i+\sum_{j\in X_2\backslash\{z_1\}} c_j\beta_j\Bigg)\]
	which holds if and only if $\sum_{i\in X_3}\beta_i\leq f$. Similarly, we get  $\varphi_{\theta'}$ is convex if and only if $\sum_{i\in X_3} \gamma_i\leq f' $ and $\varphi_{(\theta+\theta')}$ is convex if and only if $\sum_{i\in X_3} \alpha_i\leq f+f'$. As a result, if 
	\[0 \leq \sum_{i\in X_3} \alpha_i \leq f+f',\]
	we find  $\beta_J,\gamma_J \in \mathbb{Z}^{|J|}$ such that 
	$0 \leq \sum_{i\in X_3} \beta_i  \leq f$ and $0 \leq \sum_{i\in X_3} \gamma_i \leq f'$. We have seen that $\theta,\theta'$ are convex in this case.
	
	\textbf{Case 3.}
	If $p_1>p_4=1$, then we have
	\setcounter{MaxMatrixCols}{20}
	\[
	F=
	\begin{pNiceMatrix}[first-row,first-col,margin,hvlines]
		& & \Ldots[line-style={solid,<->},shorten=0pt]^{p_0+p_1-1 \text{ columns}} \\
		v_1& \Block{3-3}<\LARGE>{\blue I} & & &\Block{3-3}<\LARGE>{ \blue 0}&&\\
		\vdots &&  & &&&  \\
		v_{p_0}&& & &  \\
		u_1& \blue -1 & \blue \Cdots&\blue -1&\blue 0&\blue \Cdots&\blue 0\\
		y_1&\blue -1 &\blue \Cdots&\blue -1 &\blue -1&\blue \Cdots&\blue -1\\
		y_2&\Block{3-3}<\LARGE>{\blue 0} &&&\Block{3-3}<\LARGE>{\blue I}&&\\
		\vdots &&  & &&&  \\
		y_{p_1}&& & &  \\
	\end{pNiceMatrix}
	,
	\]

	\begin{equation*}
		\begin{split}
			\theta=
			\begin{pNiceMatrix}[first-col]
				v_1& d\\
				\vdots&\vdots\\
				v_{p_0} &0 \\
				u_1&\displaystyle{f+\sum_{i\in X_3} b_i\beta_i+\sum_{j\in X_2\backslash\{z_1\}} c_j\beta_j} \\
				y_1& \displaystyle{\sum_{i\in X_3} (b_i+1)\beta_i+\sum_{j\in X_2\backslash\{z_1\}} c_j\beta_j} \\
				y_2&0\\
				\vdots&\vdots\\
				y_{p_1}&0\\
			\end{pNiceMatrix},
		\end{split}
		\quad
		\begin{split}
			\theta'=
			\begin{pmatrix}
				d' \\
				\vdots\\
				0\\
				\displaystyle{f'+\sum_{i\in X_3} b_i\gamma_i+\sum_{j\in X_2\backslash\{z_1\}} c_j\gamma_j} \\
				\displaystyle{\sum_{i\in X_3} (b_i+1)\gamma_i+\sum_{j\in X_2\backslash\{z_1\}} c_j\gamma_j} \\
				0\\
				\vdots\\
				0
			\end{pmatrix}.
		\end{split}
	\end{equation*}
	
	In this case, we consider the fan $\Sigma$ with ray generators given by the rows of $F$ and primitive collections given by $\{v_1,\ldots,v_{p_0},u_1\}, \{y_1,\ldots,y_{p_1}\}$. The primitive relations are given by
	\begin{equation*}
		\begin{split}
			v_1+\cdots+v_{p_0}+u_1&=0\\
			y_1+\cdots+y_{p_1}&=u_1.
		\end{split}
	\end{equation*}
	It is a smooth complete fan in $\mathbb{R}^{p_0+p_1-1}$ with $p_0+p_1+1$ rays. Next, we will determine when $\varphi_{\theta},\varphi_{\theta'},\varphi_{(\theta+\theta')}$ are convex with respect to the fan $\Sigma$. The support function $\varphi_\theta$ is convex if and only if 
	
	\begin{equation*}
		\begin{split}
			\varphi_{\theta}(v_1+\cdots+v_{p_0}+u_1) &\geq \varphi_{\theta}(v_1)+\cdots\varphi_{\theta}(v_{p_0})+\varphi_{\theta}(y_1)\\
			\varphi_{\theta}( y_1+\cdots+y_{p_1})&\geq \varphi_{\theta}(y_1)+\cdots+ \varphi_{\theta}(y_{p_1}).
		\end{split}
	\end{equation*}
	
	Simplifying the first inequality gives
	\[0\geq -d -\Bigg(\sum_{i\in X_3} b_i\beta_i+\sum_{j\in X_2\backslash\{z_1\}} c_j\beta_j\Bigg), \]
	which holds in this case because $d,b_i,c_j,\beta_j$ are all non-negative. If we simplify the second inequality, we get 
	\[-\Bigg(f+\sum_{i\in X_3} b_i\beta_i+\sum_{j\in X_2\backslash\{z_1\}} c_j\beta_j\Bigg) \geq -\Bigg(\sum_{i\in X_3} (b_i+1)\beta_i+\sum_{j\in X_2\backslash\{z_1\}} c_j\beta_j\Bigg), \]
	which holds if and only if $\sum_{i\in X_3} \beta_i\geq f $. Similarly, we get  $\varphi_{\theta'}$ is convex if and only if $\sum_{i\in X_3} \gamma_i\geq f' $ and $\varphi_{(\theta+\theta')}$ is convex if and only if $\sum_{i\in X_3} \alpha_i\geq f+f'$. As a result, if 
	\[f+f' \leq \sum_{i\in X_3} \alpha_i \leq f+e+f'+e',\] 
	we find  $\beta_J,\gamma_J \in \mathbb{Z}^{|J|} $ such that 
	$f \leq \sum_{i\in X_3} \beta_i  \leq f+e$ and $f' \leq \sum_{i\in X_3} \gamma_i \leq f'+e'$. We have seen that $\theta,\theta'$ are convex in this case.
	
	But if $0\leq \sum_{j\in J}\alpha_j< f+f'$ then $\theta+\theta'$ is not convex. Nonetheless, we can see that the half-space defined by row $u_1$ is not a supporting half-space for the polytope, so we translate it without altering the polytope and construct a convex support function. Define
	\begin{equation*}
		\widetilde{\theta}_{j}=
		\begin{cases}
			\theta_{y_1}\text{~if~$j=u_1$}\\
			\theta_j\text{~otherwise,}  \end{cases}
	\end{equation*}
	
	\begin{equation*}
		\widetilde{\theta'}_{j}=
		\begin{cases}
			\theta'_{y_1}\text{~if~$j=u_1$}\\
			\theta'_j\text{~otherwise,}  \end{cases}
	\end{equation*}

	\begin{equation*}
		\widetilde{(\theta+\theta')}_{j}=
		\begin{cases}
			(\theta+\theta')_{y_1}\text{~if~$j=u_1$}\\
			(\theta+\theta')_j\text{~otherwise.}  \end{cases}
	\end{equation*}
	It easy to see that $\widetilde{\theta},\widetilde{\theta'}$ are convex and
	\begin{align*}
		\widetilde{P}&=P(F,\widetilde{\theta}), & \widetilde{Q}&=P(F,\widetilde{\theta'}), &\widetilde{P+Q}&=P(F,\widetilde{\theta}+\widetilde{\theta'}).
	\end{align*}
	
	\textbf{Case 4.}
	If $p_4>p_1=1$, then we have

	\setcounter{MaxMatrixCols}{20}
	\[
	F=
	\begin{pNiceMatrix}[first-row,first-col,margin,hvlines]
		& & \Ldots[line-style={solid,<->},shorten=0pt]^{p_0+p_4-1 \text{ columns}} \\
		v_1& \Block{3-3}<\LARGE>{\blue I} & & &\Block{3-3}<\LARGE>{ \blue 0}&&\\
		\vdots &&  & &&&  \\
		v_{p_0}&& & &  \\
		u_1& \blue -1 & \blue \Cdots&\blue -1&\blue -1&\blue \Cdots&\blue -1\\
		u_2&\Block{3-3}<\LARGE>{\blue 0} & & &\Block{3-3}<\LARGE>{\blue I}&&\\
		\vdots &&  & &&&  \\
		u_{p_4}&& & &  \\
		y_1&\blue -1 &\blue \Cdots&\blue -1 &\blue 0&\blue \Cdots& \blue 0 \\
	\end{pNiceMatrix}
	,
	\]

	\begin{equation*}
		\begin{split}
			\theta=
			\begin{pNiceMatrix}[first-col]
				v_1& d\\
				\vdots&\vdots\\
				v_{p_0} &0 \\
				u_1&\displaystyle{f+\sum_{i\in X_3} b_i\beta_i+\sum_{j\in X_2\backslash\{z_1\}} c_j\beta_j} \\
				u_2&0\\
				\vdots&\vdots\\
				u_{p_4}&0\\
				y_1& \displaystyle{\sum_{i\in X_3} (b_i+1)\beta_i+\sum_{j\in X_2\backslash\{z_1\}} c_j\beta_j}
			\end{pNiceMatrix},
		\end{split}
		\quad
		\begin{split}
			\theta'=
			\begin{pmatrix}
				d'\\
				\vdots\\
				0 \\
				\displaystyle{f'+\sum_{i\in X_3} b_i\gamma_i+\sum_{j\in X_2\backslash\{z_1\}} c_j\gamma_j} \\
				0\\
				\vdots\\
				0\\
				\displaystyle{\sum_{i\in X_3} (b_i+1)\gamma_i+\sum_{j\in X_2\backslash\{z_1\}} c_j\gamma_j} \\
			\end{pmatrix}.
		\end{split}
	\end{equation*}
	
	In this case, we consider the fan $\Sigma$ with ray generators given by the rows of $F$ and primitive collections given by $\{v_1,\ldots,v_{p_0},y_1\}, \{u_1,\ldots,u_{p_4}\}$. The primitive relations are given by
	\begin{equation*}
		\begin{split}
			v_1+\cdots+v_{p_0}+y_1&=0\\
			u_1+\cdots+u_{p_4}&=y_1.
		\end{split}
	\end{equation*}
	It is a smooth complete fan in $\mathbb{R}^{p_0+p_4-1}$ with $p_0+p_4+1$ rays. Next, we will determine when $\varphi_{\theta},\varphi_{\theta'},\varphi_{(\theta+\theta')}$ are convex with respect to the fan $\Sigma$. The support function $\varphi_\theta$ is convex if and only if 
	
	\begin{equation*}
		\begin{split}
			\varphi_{\theta}(v_1+\cdots+v_{p_0}+y_1) &\geq \varphi_{\theta}(v_1)+\cdots\varphi_{\theta}(v_{p_0})+\varphi_{\theta}(y_1)\\
			\varphi_{\theta}( u_1+\cdots+u_{p_4})&\geq \varphi_{\theta}(u_1)+\cdots+ \varphi_{\theta}(u_{p_4}).
		\end{split}
	\end{equation*}
	
	Simplifying the first inequality gives
	\[0\geq -d -\Bigg(\sum_{i\in X_3} (b_i+1)\beta_i+\sum_{j\in X_2\backslash\{z_1\}} c_j\beta_j\Bigg), \]
	which holds in this case because $d,b_i,c_j,\beta_j$ are all non-negative. If we simplify the second inequality, we get 
	\[-\Bigg(\sum_{i\in X_3} (b_i+1)\beta_i+\sum_{j\in X_2\backslash\{z_1\}} c_j\beta_j\Bigg) \geq -\Bigg(f+\sum_{i\in X_3} b_i\beta_i+\sum_{j\in X_2\backslash\{z_1\}} c_j\beta_j\Bigg), \]
	which holds if and only if $\sum_{i\in X_3} \beta_i\leq f $. Similarly, we get  $\varphi_{\theta'}$ is convex if and only if $\sum_{i\in X_3} \gamma_i\leq f' $ and $\varphi_{(\theta+\theta')}$ is convex if and only if $\sum_{i\in X_3} \alpha_i\leq f+f'$. As a result, if 
	\[0 \leq \sum_{i\in X_3} \alpha_i \leq f+f',\]
	we find  $\beta_J,\gamma_J \in \mathbb{Z}^{|J|} $ such that 
	$0 \leq \sum_{i\in X_3} \beta_i  \leq f$ and $0 \leq \sum_{i\in X_3} \gamma_i \leq f'$. We have seen that $\theta,\theta'$ are convex in this case.
	
	But if $f+f'< \sum_{j\in J}\alpha_j\leq f+e+f'+e'$ then $\theta+\theta'$ is not convex. Nonetheless, we can see that the half-space defined by row $y_1$ is not a supporting half-space for the polytope, so we translate it without altering the polytope and construct a convex support function. Define 
	\begin{equation*}
		\widetilde{\theta}_{j}=
		\begin{cases}
			\theta_{u_1}\text{~if~$j=y_1$}\\
			\theta_j\text{~otherwise,}  \end{cases}
	\end{equation*}
	
	\begin{equation*}
		\widetilde{\theta'}_{j}=
		\begin{cases}
			\theta'_{u_1}\text{~if~$j=y_1$}\\
			\theta'_j\text{~otherwise,}  \end{cases}
	\end{equation*}
	
	\begin{equation*}
		\widetilde{(\theta+\theta')}_{j}=
		\begin{cases}
			(\theta+\theta')_{u_1}\text{~if~$j=y_1$}\\
			(\theta+\theta')_j\text{~otherwise.}  \end{cases}
	\end{equation*}
	It easy to see that $\widetilde{\theta},\widetilde{\theta'}$ are convex and
	\begin{align*}
		\widetilde{P}&=P(F,\widetilde{\theta}), & \widetilde{Q}&=P(F,\widetilde{\theta'}), &\widetilde{P+Q}&=P(F,\widetilde{\theta}+\widetilde{\theta'}).
	\end{align*}
\end{proof}	

\begin{remark}
	We can use the strategy of elimination variables together with induction to give a different proof for the fact that IDP holds for any smooth complete splitting fan.
\end{remark}

\bibliographystyle{alpha}
\bibliography{references} 	
\end{document}